\documentclass[12pt]{amsart}
\usepackage{amsmath}\usepackage{amscd}\usepackage{amssymb}\usepackage{amsfonts}
\newtheorem{theorem}{Theorem}[section]
\newtheorem{lemma}[theorem]{Lemma}
\newtheorem{coro}[theorem]{Corollary}

\theoremstyle{definition}

\theoremstyle{remark}
\newtheorem{remark}[theorem]{Remark}
\numberwithin{equation}{section}

\begin{document}
\title[Gradient estimates and Liouville theorems]
{Gradient estimates for a nonlinear parabolic equation and Liouville theorems}

\author{Jia-Yong Wu}
\address{Department of Mathematics, Shanghai Maritime University, Shanghai 201306, P. R. China}
\email{jywu81@yahoo.com}

\subjclass[2010]{Primary 53C21, 58J35; Secondary 35B53, 35K55}
\date{Received: September 05, 2017; Revised: July 22, 2018.}
\keywords{Nonlinear parabolic equation; gradient estimate;
Liouville theorem; Yamabe-type problem; smooth metric measure space;
Bakry-\'{E}mery Ricci tensor.}
\begin{abstract}
We establish local elliptic and parabolic gradient estimates for positive smooth solutions
to a nonlinear parabolic equation on a smooth metric measure space. As applications, we
determine various conditions on the equation's coefficients and the growth of solutions
that guarantee the nonexistence of nontrivial positive smooth solutions to many special
cases of the nonlinear equation. In particular, we apply gradient estimates to discuss
some Yamabe-type problems of complete Riemannian manifolds and smooth metric measure spaces.
\end{abstract}
\maketitle

\section{Introduction and main results}\label{Int1}
In this paper we will study gradient estimates for positive smooth solutions $u(x,t)$ to
a parabolic equation
\begin{equation}\label{Yameq}
\left(\Delta_f-\frac{\partial}{\partial t}\right)u+\mu(x,t)u+p(x,t)u^\alpha+q(x,t)u^\beta=0
\end{equation}
on a smooth metric measure space $(M,g,e^{-f}dv_g)$, where $\mu(x,t)$, $p(x,t)$ and $q(x,t)$
are all smooth space-time functions, and $\alpha,\beta\in\mathbb{R}$.
As applications, we give Liouville-type theorems for various special cases of the equation
\eqref{Yameq}. In particular, since the equation \eqref{Yameq} is related to Yamabe-type
problems (see the explanation below), we also apply gradient estimates to study some
Yamabe-type problems of complete Riemannian manifolds and smooth metric measure spaces.

\vspace{.1in}

A smooth metric measure space is a tuple $(M,g,e^{-f}dv_g)$ of an $n$-dimensional complete
Riemannian manifold $(M,g)$, and a weighted measure $e^{-f}dv_g$ determined by some
$f\in C^\infty(M)$ and the Riemannian volume element $dv_g$ of the metric $g$. Such spaces
arise in many contexts, for example as collapsed measured Gromov-Hausdorff limits \cite{[Lott1]}.
On $(M,g,e^{-f}dv_g)$, the $f$-Laplacian is defined by
\[
\Delta_f=\Delta-\nabla f\cdot\nabla,
\]
where $\Delta$ is the usual Laplacian, which is self-adjoint with respect to $e^{-f}dv_g$.
For any number $m\geq 0$, the $m$-Bakry-\'Emery Ricci tensor introduced by Bakry and
\'Emery \cite{[BE]} is defined by
\[
\mathrm{Ric}^m_f:=\mathrm{Ric}+\mathrm{Hess}\, f-\frac{1}{m}df\otimes df,
\]
where $\mathrm{Ric}$ is the Ricci tensor of $(M,g)$, and $\mathrm{Hess}$ is the
Hessian of metric $g$. When $m=0$, it means that $f$ is constant and $\mathrm{Ric}^m_f$
returns to the usual Ricci tensor $\mathrm{Ric}$. In \cite{[Lott2]}, the weighted scalar
curvature related to $\mathrm{Ric}^m_f$ is defined by
\[
\mathrm{S}^m_f:=\mathrm{S}+2 \Delta f-\frac{m+1}{m}|\nabla f|^2,
\]
where $\mathrm{S}$ is the scalar curvature of $(M,g)$. In general, $\mathrm{S}^m_f$
is not the trace of $\mathrm{Ric}^m_f$, except when $f$
is constant. When $m\to\infty$, we have the Perelman's scalar curvature
(see \cite{[Per]})
\[
\mathrm{S}^{\infty}_f:=\mathrm{S}+2 \Delta f-|\nabla f|^2
\]
and the $(\infty-)$Bakry-\'Emery Ricci tensor
\[
\mathrm{Ric}_f:=\mathrm{Ric}^\infty_f.
\]
It is easy to see that $\mathrm{Ric}^m_f\geq c$ implies $\mathrm{Ric}_f\geq c$,
but not vice versa.

On a smooth metric measure space $(M,g,e^{-f}dv_g)$, if
\[
\mathrm{Ric}_f=\lambda\, g
\]
for some $\lambda\in\mathbb{R}$, then $(M,g,e^{-f}dv_g)$ is a gradient Ricci soliton,
which is a generalization of an Einstein manifold. Gradient
Ricci solitons play a fundamental role in the formation of singularities of the Ricci
flow, and have been studied by many authors; see \cite{[Cao],[Ham]} and references
therein for nice surveys.

\vspace{.1in}

There have been many gradient estimates and Liouville-type theorems about special
cases of the equation \eqref{Yameq}. In 1980s, B. Gidas and J. Spruck \cite{[GS]}
studied the equation
\begin{equation}\label{equa1}
\Delta u+p(x)u^\alpha=0,\quad1\leq\alpha<\frac{n+2}{n-2}
\end{equation}
on an $n$-dimensional manifold. The case $\alpha=3$ is relevant to Yang-Mills equations
(see \cite{[CGS]}). The case $\alpha<0$ is related to a steady state of the thin film
(see \cite{[GW]}). B. Gidas and J. Spruck \cite{[GS]} proved that any nonnegative solution
to the equation \eqref{equa1} is identically zero when the Ricci tensor of manifold
is nonnegative. Y. Yang \cite{[Yang]} showed that if $\alpha<0$ and $p(x)$ is positive
constant, then the equation \eqref{equa1} does not admit any positive solution on
a complete manifold with the nonnegative Ricci tensor. J.-Y. Li \cite{[JLi]} proved the
Gidas-Spruck's result under some weaker restrictions of $p(x)$ for $1<\alpha<\frac{n}{n-2}$
($n\geq 4$). He also proved Li-Yau gradient estimates and Harnack
inequalities for the nonlinear parabolic equation
\begin{equation}\label{equa2}
\left(\Delta-\frac{\partial}{\partial t}\right)u+p(x,t)u^\alpha=0,\quad \alpha>0
\end{equation}
on a manifold. In biomathematics, the equation \eqref{equa2} could be interpreted as the
population dynamics (see \cite{[CaCo]}). Recently, X. Zhu \cite{[Zhu1],[Zhu2]} gave
elliptic gradient estimates and Liouville-type theorems for positive ancient solutions
to the equation \eqref{equa2}.

Apart from the relation to the above equations, the famous and widely studied special example
of the equation \eqref{Yameq} is related to conformally deformation of the scalar curvature
on a manifold. Indeed, for any $n$-dimensional ($n\geq 3$) complete manifold $(M,g)$, consider
a pointwise conformal metric
\[
\tilde{g}=u^{\frac{4}{n-2}}g
\]
for some $0<u\in C^\infty(M)$. Then the scalar curvature $\mathrm{\tilde{S}}$
of metric $\tilde{g}$ related to the scalar curvature $\mathrm{S}$ of metric $g$ is given by
(see \cite{[MRS]})
\begin{equation}\label{maineq}
\Delta u-\frac{n-2}{4(n-1)}\mathrm{S}\,u+\frac{n-2}{4(n-1)}\mathrm{\tilde{S}}\,u^{\frac{n+2}{n-2}}=0,
\end{equation}
which is a special form of equation \eqref{Yameq}. If $M$ is compact and $\mathrm{\tilde{S}}$
is constant, the existence of a positive solution $u$ is the well-known Yamabe problem and
it has been solved in the affirmative by the combined efforts of Yamabe \cite{[Yama]},
Trudinger \cite{[Tru]}, Aubin \cite{[Aub]} and Schoen \cite{[Sc84]}; see the survey
\cite{[LP]} for more details. However, if $M$ is noncompact ($\mathrm{\tilde{S}}$ is still
constant), Z. Jin \cite{[Jin]} gave examples of complete metrics on the noncompact manifold
on which there do not exist a positive smooth solution of \eqref{maineq}. When $\mathrm{\tilde{S}}$
is a smooth function, the geometry of manifolds plays a large role in the
existence and nonexistence of positive solutions of \eqref{maineq} on compact or noncompact manifolds.
The interested reader can refer to \cite{[Sch],[J93],[Bis],[LTY],[Zhang],[MRS]} and references therein.

\vspace{.1in}

Another important reason of studying the equation \eqref{Yameq} is that a static form of
equation \eqref{Yameq} is related to the weighted Yamabe problem posed by J. Case \cite{[Case]}.
Recall that, for any $m\ge0$, Case \cite{[Case]} introduced the weighted Yamabe quotient
\[
\mathcal{Q}(u):= \frac{\left(\int_M|\nabla u|^2+\frac{m+n-2}{4(m+n-1)}\mathrm{S}^m_f u^2\right)
\left(\int_M|u|^{\frac{2(m+n-1)}{m+n-2}}e^{\frac fm}\right)^{\frac{2m}{n}}}{\left(\int_M |u|^{\frac{2(m+n)}{m+n-2}}\right)^{\frac{2m+n-2}{n}}}
\]
on a smooth metric measure space $(M,g,e^{-f}dv_g)$, where all integrals are taken with
respect to the weighted measure $e^{-f}dv_g$. The weighted Yamabe quotient is conformally
invariant in the sense that if
\[
\left(M^n,\tilde{g}, e^{-\tilde{f}} dv_{\tilde{g}}\right)
=\left(M^n,e^{\frac{2\rho}{m+n-2}}g, e^{\frac{(m+n)\rho}{m+n-2}}e^{-f}dv_g\right)
\]
for some $\rho\in C^\infty(M)$, then $\mathcal{\tilde{Q}}(u)=\mathcal{Q}(e^{\frac{\rho}{2}}u)$
(see \cite{[Case]}). The weighted Yamabe constant is defined by
\[
\Lambda[g, e^{-f}dv_g]:=\inf\left\{\mathcal{Q}(u)|\,\,0<u\in C^\infty(M)\right\},
\]
which is a generalization of the Yamabe constant. Indeed, if $f=0$ and $m=0$, the
weighted Yamabe constant returns to the classical Yamabe constant. In \cite{[Case]} Case
observed that $u$ is a critical point of the weighted Yamabe quotient
$\mathcal{Q}(u)$ on a smooth metric measure space $(M,g,e^{-f}dv_g)$ if and only if it satisfies
\begin{equation}\label{weYaeq}
\Delta_f u-\frac{m+n-2}{4(m+n-1)}\mathrm{S}^m_fu-c_1e^{\frac fm}u^{\frac{m+n}{m+n-2}}+c_2u^{\frac{m+n+2}{m+n-2}}=0,
\end{equation}
which is a special elliptic case of \eqref{Yameq} in some setting. Here,
\begin{align*}
c_1 & = \frac{2m(m+n-1)\mathcal{Q}(u)}{n(m+n-2)}\left(\int_M u^{\frac{2(m+n)}{m+n-2}}\right)^{\frac{2m+n-2}{n}}
\left(\int_M u^{\frac{2(m+n-1)}{m+n-2}}e^{\frac fm}\right)^{-\frac{2m+n}{n}} , \\
c_2 & = \frac{(2m+n-2)(m+n)\mathcal{Q}(u)}{n(m+n-2)}\left(\int_M u^{\frac{2(m+n)}{m+n-2}}\right)^{\frac{2m-2}{n}}
\left(\int_M u^{\frac{2(m+n-1)}{m+n-2}}e^{\frac fm}\right)^{-\frac{2m}{n}},
\end{align*}
where all integrals are taken with respect to $e^{-f}dv_g$. Obviously, $c_1$ and $c_2$
have the same sign. When $\Lambda[g, e^{-f}dv_g]=0$, we have $c_1=c_2=0$ and the
critical point of $\mathcal{Q}$ is in fact a minimizer of $\Lambda$. Case \cite{[Case]}
proved that minimizers always exist on a compact smooth metric measure space provided the
weighted Yamabe constant is strictly less than its value on Euclidean space.

\vspace{.1in}

In this paper, we will give local elliptic and parabolic gradient estimates for positive
solutions to the equation \eqref{Yameq} on a smooth metric measure space with the
Bakry-\'Emery Ricci tensor bounded below. As applications, we will determine various
conditions on the growth of solutions and coefficients that guarantee the nonexistence
of nontrivial positive smooth solutions to many special cases of the equation
\eqref{Yameq}. In particular, we can apply gradient estimates to analyze Yamabe-type
problems of equations \eqref{maineq} and \eqref{weYaeq} on a complete manifold and a smooth
metric measure space, respectively.

\vspace{.1in}

In order to state the results, we introduce some notations. On an $n$-dimensional complete
smooth metric measure space $(M,g,e^{-f}dv)$, let $\nabla$ and $|\cdot|$ stand for the
Levi-Civita connection and the norm with respect to metric $g$, respectively. For a fixed point
$x_0\in M$ and $R>0$, let $r(x)$ (or $d(x,x_0)$) denote a distance function to $x$ from $x_0$
with respect to $g$, and $B(x_0,R)$ denote the geodesic ball centered at $x_0$ of radius $R$.
In the elliptic gradient estimate setting, let $Q_{R,T}$ be
\[
Q_{R,T}:\equiv B(x_0,R)\times[t_0-T,t_0]\subset M\times(-\infty,\infty),\quad t_0\in \mathbb{R}\,\,\mathrm{and}\,\,T>0.
\]
In the parabolic gradient estimate setting, let $H_{R,T}$ be
\[
H_{R,T}:\equiv B(x_0,R)\times[0,T],\quad T>0.
\]
For any $\mu\in C^\infty(Q_{R,T})$, denote
\[
\mu^+:=\sup_{(x,t)\in Q_{R,T}}\{\mu^+(x,t),0\}\quad \mathrm{and}\quad
\mu^-:=\inf_{(x,t)\in Q_{R,T}}\{\mu^-(x,t),0\},
\]
where $\mu^+(x,t):=\max\{\mu(x,t),0\}$ and $\mu^-(x,t):=\min\{\mu(x,t),0\}$.
For $\mu\in C^\infty(H_{R,T})$, we similarly define $\mu^+$ and $\mu^-$
in $H_{R,T}$ as above. We also introduce the geometric quantities
\[
\sigma:=\max_{\{x|d(x,x_0)=1\}}\Delta_f\,r(x)\quad \mathrm{and}\quad \sigma^+:=\max\{\sigma,0\},
\]
which will appear in our theorems.

\vspace{.1in}

We now give one of main theorems, a local elliptic (space-only) gradient estimates for
positive smooth solutions to the equation \eqref{Yameq} when $\mathrm{Ric}_f$ is bounded below.
\begin{theorem}\label{main}
Let $(M,g,e^{-f}dv)$ be an $n$-dimensional complete smooth metric measure space.
Assume that $\mathrm{Ric}_f\geq-(n-1)K$ for some constant $K\geq0$ in $B(x_0,R)$,
where $x_0\in M$ and $R\geq2$.  Let $0<u(x,t)\leq D$ for some constant $D$, be a
smooth solution to the equation \eqref{Yameq} in $Q_{R,T}:=B(x_0,R)\times[t_0-T,t_0]$.
There exists a constant $c$ depending only on $n$, such that
\begin{equation*}
\begin{aligned}
|\nabla\ln u|\leq c\left(1+\ln \frac Du\right)&
\Bigg[\frac 1R+\sqrt{\frac{\sigma^+}{R}}+\frac{1}{\sqrt{t-t_0+T}}+\sqrt{K}
+\sqrt{\mu^+}+\sup_{Q_{R,T}}\kern-4pt|\nabla\mu|^{\frac 13}\\
&+\sqrt{[(\alpha-1)p]^++p^+}\,\sup_{Q_{R,T}}\{u^{\frac{\alpha-1}{2}}\}
+\sup_{Q_{R,T}}|\nabla p|^{\frac 13}\sup_{Q_{R,T}}\{u^{\frac{\alpha-1}{3}}\}\\
&+\sqrt{[(\beta-1)q]^++q^+}\,\sup_{Q_{R,T}}\{u^{\frac{\beta-1}{2}}\}
+\sup_{Q_{R,T}}|\nabla q|^{\frac 13}\sup_{Q_{R,T}}\{u^{\frac{\beta-1}{3}}\}\Bigg]
\end{aligned}
\end{equation*}
in $Q_{R/2, T}$ with $t\neq t_0-T$.
\end{theorem}
\begin{remark}\label{rem1}
If $f$ is constant, the term $\sqrt{\frac{\sigma^+}{R}}$ is unnecessary in the above
estimate. If $\mu(x,t)$, $p(x,t)$ and $q(x,t)$ are identically zero, the theorem returns
to \cite{[Wu15]}. Recently, N.-T. Dung et al. \cite{[DKN]} proved similar results when
$\mu(x,t)$, $p(x,t)$, $q(x,t)$, $\alpha$ and $\beta$ are special constants.
\end{remark}

\medskip

Besides, we can give a local parabolic (space-time) gradient estimate for positive smooth
solutions to the equation \eqref{Yameq} when $\mathrm{Ric}^m_f$ is bounded below.
\begin{theorem}\label{mainpara}
Let $(M,g,e^{-f}dv)$ be an $n$-dimensional complete smooth metric measure space.
Assume that $\mathrm{Ric}^m_f\geq-(m+n-1)K$ ($m<\infty$) for some constant $K\geq0$ in $B(x_0,2R)$,
where $x_0\in M$ and $R>0$.  Let $u(x,t)$ be a positive smooth solution to the equation
\eqref{Yameq} in $H_{2R,T}:=B(x_0,2R)\times[0,T]$. Also assume that
\[
|\nabla p|\leq a_1,\,\,\,\Delta_fp\geq b_1\,\,\,\mathrm{for\,\,\,some\,\,\,constants}\,\,\, a_1\,\,\,\mathrm{and}\,\,\,b_1;
\]
\[
|\nabla q|\leq a_2,\,\,\,\Delta_fq\geq b_2\,\,\,\mathrm{for\,\,\,some\,\,\,constants}\,\,\, a_2\,\,\,\mathrm{and}\,\,\,b_2;
\]
\[
|\nabla\mu|\leq a_3,\,\,\,\Delta_f\mu\geq b_3\,\,\,\mathrm{for\,\,\,some\,\,\,constants}\,\,\, a_3\,\,\,\mathrm{and}\,\,\,b_3
\]
in $B(x_0,2R)$. For any $\lambda>1$ and $\varepsilon\in(0,1)$ satisfying $\Psi\geq 0$,
there exists a universal positive constant $c_1$ independent of the geometry of $(M,g,e^{-f}dv)$
such that
\begin{equation*}
\begin{aligned}
\frac{|\nabla u|^2}{\lambda u^2}+pu^{\alpha-1}&+qu^{\beta-1}+\mu
-\frac{u_t}{u}\leq\frac{(m+n)\lambda}{2t}+\sqrt{\frac{m+n}{2}}\Psi^{\frac12}\\
&+\frac{m+n}{2R^2}\lambda\left[(m+n)c_1(1+R\sqrt{K})+2c^2_1
+\frac{(m+n)c^2_1\lambda^2}{4(\lambda-1)}\right]\\
&+\frac{m+n}{2}\lambda\left\{\big[(\alpha-1)p\big]^+\sup_{H_{2R,T}}\{u^{\alpha-1}\}+\big[(\beta-1)q\big]^+
\sup_{H_{2R,T}}\{u^{\beta-1}\}\right\}
\end{aligned}
\end{equation*}
in $B(x_0,R)\times(0,T]$, where
\[
\Psi:=\frac32\left[\frac{(m+n)\lambda^2}{4\varepsilon(\lambda-1)^2}\right]^{\frac13}\gamma^{\frac43}
+\frac{(m+n)\lambda^2\tilde{K}^2}{2(1-\varepsilon)(\lambda-1)^2}
-\lambda\left[\inf_{H_{2R,T}}\left(u^{\alpha-1}b_1+u^{\beta-1}b_2\right)+b_3\right],
\]
\[
\gamma:=a_1\,|\lambda\alpha-1|\sup_{H_{2R,T}}\{u^{\alpha-1}\}+a_2\,|\lambda\beta-1|\sup_{H_{2R,T}}\{u^{\beta-1}\}+a_3(\lambda-1)
\]
and
\[
\tilde{K}{:=}(m+n-1)K-\frac 12\big[(\alpha-1)(\lambda\alpha-1)p\big]^-\kern-3pt\sup_{H_{2R,T}}\{u^{\alpha-1}\}
-\frac 12\big[(\beta-1)(\lambda\beta-1)q\big]^-\kern-3pt\sup_{H_{2R,T}}\{u^{\beta-1}\}.
\]
\end{theorem}
\begin{remark}
If $f$ is constant, $p(x,t)$ and $q(x,t)$ are identically zero, then the theorem returns
to the well-known Li-Yau gradient estimate \cite{[Li-Yau]}. More parabolic gradient estimates for
special cases of the equation \eqref{Yameq} were proved in \cite{[CCK],[CLPW],[JLi],[LD]}.
\end{remark}

Theorem \ref{main} describes local elliptic gradient estimates under only $\mathrm{Ric}_f$
bounded below, whose assumption on $\mathrm{Ric}_f$ is obviously weaker than the assumption
on $\mathrm{Ric}^m_f$ ($m<\infty$). Theorem \ref{mainpara} describes local Li-Yau gradient
estimates under the assumption on $\mathrm{Ric}^m_f$ ($m<\infty$) rather than $\mathrm{Ric}_f$,
because, according to \cite{[Wu15]}, there seems essential obstacles to obtain Li-Yau gradient
estimates for the equation \eqref{Yameq} when $\mathrm{Ric}_f$ is bounded below, even assuming
growth assumption on $f$.

\medskip

Theorems \ref{main} and \ref{mainpara} have many applications. On one hand, we apply
Theorem \ref{main} to get parabolic Liouville-type theorems for special cases of the
equation \eqref{Yameq}. Here, we only provide two typical results. More related results
will be discussed in Section \ref{sec4}.

\medskip

\begin{theorem}\label{app1}
Let $(M,g,e^{-f}dv)$ be an $n$-dimensional complete smooth metric measure space with
$\mathrm{Ric}_f\geq0$. Assume that there exist two constants $s>0$ and $\kappa>0$,
such that $\mu(x)$ and $p(x)$ in the following equation
\begin{equation}\label{Yameqapp1}
\left(\Delta_f-\frac{\partial}{\partial t}\right)u+\mu(x)u+p(x)u^\alpha=0,
\quad \alpha>1,\,\,p(x)\not\equiv0,
\end{equation}
satisfy
\begin{enumerate}
\item $\mu^+\big|_{B(x_0,R)}=o(R^{-s})$ and $\sup_{B(x_0,R)}|\nabla\mu|=o(R^{-s})$,
\, as\, $R\to\infty$;
\item $p^+|_{B(x_0,R)}=o[R^{-\kappa(\alpha-1)}]$ and $\sup_{B(x_0,R)}|\nabla p|=o[R^{-\kappa(\alpha-1)}]$,
\, as\, $R\to\infty$.
\end{enumerate}
Let $u(x,t)$ be a positive ancient solution to the equation \eqref{Yameqapp1}
(that is, a solution defined in all space and negative time) such that
\[
u(x,t)=o[(r(x)+|t|)^{\widetilde{\kappa}}]
\]
for some $\widetilde{\kappa}\in(0,\kappa)$ near infinity.
Then $u(x,t)\equiv c^{\frac{1}{\alpha-1}}$ and $\mu(x)\equiv-cp(x)$ for some
constant $c>0$.
\end{theorem}

\begin{theorem}\label{app3}
Let $(M,g,e^{-f}dv)$ be an $n$-dimensional complete smooth metric measure space with $\mathrm{Ric}_f\geq 0$.
Assume that there exist two constants $s>0$ and $\kappa>0$,
such that $\mu(x)$ and $p(x)$ in the following equation
\begin{equation}\label{Yameqapp3}
\left(\Delta_f-\frac{\partial}{\partial t}\right)u+\mu(x)u+p(x)u^\alpha=0,
\quad \alpha<1,\,\,p(x)\not\equiv0,
\end{equation}
satisfy
\begin{enumerate}
\item $\mu^+\big|_{B(x_0,R)}=o(R^{-s})$ and $\sup_{B(x_0,R)}|\nabla\mu|=o(R^{-s})$,
\, as\, $R\to\infty$;
\item $\sup_{B(x_0,R)}|\,p\,|=o[R^{-\kappa(1-\alpha)}]$ and $\sup_{B(x_0,R)}|\nabla p|=o[R^{-\kappa(1-\alpha)}]$,
\, as\, $R\to\infty$.
\end{enumerate}
Let $u(x,t)$ be a positive ancient solution to \eqref{Yameqapp3} such that
\[
(r(x)+|t|)^{-\widetilde{\kappa}}\leq u(x,t)\leq(r(x)+|t|)^{\delta}
\]
for some $\widetilde{\kappa}\in(0,\kappa)$ and $\delta>0$ near infinity.
Then $u(x,t)\equiv c^{\frac{1}{\alpha-1}}$ and $\mu(x)\equiv-cp(x)$ for some
constant $c>0$.
\end{theorem}

\medskip

We also apply Theorem \ref{main} to prove Liouville-type theorems for
elliptic versions of the equation \eqref{Yameq}; see for example Theorems
\ref{ellap1} and \ref{ellap2} in Section \ref{sec5}. In
particular, we apply Theorem \ref{app1} to study the problem about
conformal deformation of the scalar curvature on complete manifolds.
\begin{theorem}\label{app4}
Let $(M,g)$ be an $n$-dimensional $(n\geq 3)$ complete (possible noncompact) Riemannian manifold
with $\mathrm{Ric}\geq 0$ and $\sup_{B(x_0,R)}|\nabla\mathrm{S}|=o(R^{-s})$ for some
constant $s>0$, as $R\to\infty$. For any $\kappa>0$, there does not exist complete metric
\[
\tilde{g}\in\left\{u^{\frac{4}{n-2}}g\,\big|\,0<u\in C^\infty(M), u(x)=o(r^{\widetilde{\kappa}}(x))\right\}
\]
for some $\widetilde{\kappa}\in(0,\kappa)$, such that the scalar curvature $\mathrm{\tilde{S}}$ of $\tilde{g}$ satisfies
\[
\mathrm{\tilde{S}}^+\big|_{B(x_0,R)}=o(R^{-\frac{4\kappa}{n-2}})\quad\mathrm{and}
\quad\sup_{B(x_0,R)}|\nabla\mathrm{\tilde{S}}|=o(R^{-\frac{4\kappa}{n-2}}),\quad\mathrm{as}\,\,R\to\infty.
\]
\end{theorem}
\begin{remark}
If $\mathrm{\tilde{S}}$ is nonpositive constant, the growth conditions of $\mathrm{\tilde{S}}$
and $\nabla\mathrm{\tilde{S}}$ in Theorem \ref{app4} naturally hold and hence $u(x)$ can
be relaxed to $u(x)=e^{o(r^{\frac s3}(x))}$. Compared with the work of \cite{[MRS],[RRV]},
Theorem \ref{app4} is valid without any assumptions on sectional curvature, eigenvalue
of the conformal operator $\Delta-\frac{n-2}{4(n-1)}\mathrm{S}$, only assuming some
conditions of the Ricci tensor and the growth of $u(x)$.
\end{remark}

On a compact smooth metric measure space, Case \cite{[Case]} provided an example
which shows that minimizers of the weighted Yamabe constant do not always exist.
Using Theorem \ref{main} we can prove

\begin{theorem}\label{generalcase}
Let $(M,g,e^{-f}dv)$ be an $n$-dimensional $(n\geq 3)$ complete smooth metric measure
space with $\mathrm{Ric}_f\geq 0$. For any $m>0$, assume that there exist two
constants $s>0$ and $\kappa>0$ such that
\begin{enumerate}
\item $(\mathrm{S}^m_f)^-\big|_{B(x_0,R)}=o(R^{-s})$\, and\, $\sup_{B(x_0,R)}|\nabla\mathrm{S}^m_f|=o(R^{-s})$,
\, as\, $R\to\infty$;
\item $e^{\frac fm}|_{B(x_0,R)}=o[R^{\frac{-2\kappa}{m+n-2}}]$ and $\sup_{B(x_0,R)}|\nabla e^{\frac fm}|=o[R^{\frac{-2\kappa}{m+n-2}}]$,
\, as\, $R\to\infty$.
\end{enumerate}
Then there does not exist a minimizer of the weighted Yamabe constant $\Lambda\leq0$ with $u(x)=o\,(r^{\widetilde{\kappa}}(x)\,)$ for some $\widetilde{\kappa}\in(0,\kappa)$ near infinity.
\end{theorem}

When $\Lambda=0$, we have a simple statement.
\begin{theorem}\label{app5}
Let $(M,g,e^{-f}dv)$ be an $n$-dimensional $(n\geq 3)$ complete smooth
metric measure space with $\mathrm{Ric}_f\geq 0$. For any $m>0$, assume that
\[
(\mathrm{S}^m_f)^-\big|_{B(x_0,R)}=o(R^{-1})\quad\mathrm{and}\quad
\sup_{B(x_0,R)}|\nabla\mathrm{S}^m_f|=o(R^{-\frac 32}),\quad\mathrm{as}\,\,R\to\infty.
\]
If the weighted Yamabe constant $\Lambda=0$, there does not exist a critical point of
the weighted Yamabe quotient $\mathcal{Q}(u)$ with $u(x)=e^{o(r^{1/2}(x))}$ near infinity.
\end{theorem}

\medskip

On the other hand, we can apply Theorem \ref{mainpara} to give a new Liouville
theorem for an elliptic case of the equation \eqref{Yameq}, which is a supplement
to Yang's result \cite{[Yang]}.
\begin{theorem}\label{ellliou}
Let $(M,g,e^{-f}dv)$ be an $n$-dimensional complete smooth metric measure space
with $\mathrm{Ric}^m_f\geq0$. Then there does not exist any nontrivial positive
solution $u(x)$ to the elliptic equation
\begin{equation}\label{elli}
\Delta_f u+p u^\alpha=0, \quad \alpha\leq1,
\end{equation}
where $p$ is a nonnegative constant.
\end{theorem}

When $\Lambda=0$, Theorem \ref{ellliou} indeed implies that
\begin{coro}\label{app6}
Let $(M,g,e^{-f}dv)$ be an $n$-dimensional $(n\geq 3)$ complete smooth metric measure
space with $\mathrm{Ric}^m_f\geq 0$. Assume that the weighted scalar curvature
$\mathrm{S}^m_f$ is nonpositive constant. If the weighted Yamabe constant $\Lambda=0$,
there does not exist a critical point of the weighted Yamabe quotient $\mathcal{Q}$.
\end{coro}
\begin{remark}
In view of Theorem \ref{generalcase}, we may apply Theorem \ref{mainpara} to study the minimizer
of the weighted Yamabe constant $\Lambda\le0$ (or $\Lambda\ge0$). This enables us to determine many
complicated assumptions so that we can apply the Li-Yau gradient estimate of Theorem \ref{mainpara}
to achieve the Liouville-type theorem for the equation \eqref{weYaeq}. In the paper we do not
describe this complicated case.
\end{remark}

\medskip

Inequalities in Theorem \ref{main} and Theorem \ref{mainpara} are called local elliptic and parabolic
gradient estimates, respectively (sometimes called Hamilton-Souplet-Zhang and Li-Yau gradient
estimates, respectively), which are both proved by using the maximum principle in a locally
supported set of the manifold. Similar inequalities have been obtained for the linear heat
equation, e.g. \cite{[ChHam],[LXu],[Li-Yau],[LD],[Sou-Zh],[Wu15]} and some nonlinear equations,
e.g. \cite{[CaZh],[DKN],[JLi],[Ya08],[Zhu1],[Zhu2]}. However, our case is more complicated
due to the function coefficients of equation \eqref{Yameq}. To the best of our knowledge,
the gradient estimate technique is originated by Yau \cite{[Yau]} (see also Cheng-Yau
\cite{[Cheng-Yau]}) in 1970s, who first proved a gradient estimate for the harmonic
function on the manifold. In 1980s, this technique was developed by Li-Yau \cite{[Li-Yau]}
for the heat equation on manifolds (though a precursory form of their estimate appeared
in \cite{[AB]}). In 1990s, R. Hamilton \cite{[Ham93]} gave an elliptic gradient estimate
for the heat equation. But this estimate is global which requires the equation defined
on closed manifolds. In 2006, Souplet and Zhang \cite{[Sou-Zh]} proved a local elliptic
form by adding a logarithmic correction term. Recently, many authors extended the Li-Yau
and Hamilton-Souplet-Zhang gradient estimates to the other heat-type equations; see for
example \cite{[Ba],[CCK],[CFL],[CLPW],[DKN],[Wu15],[Wu17],[Xu],[Zhu1],[Zhu2]} and
references therein.

\medskip

The paper is organized as follows. In Section \ref{sec2}, we first
give a useful lemma. Then we apply the lemma and the maximum principle to prove
Theorem \ref{main}. In Section \ref{sec3}, we start to give a lemma, and
then we apply the lemma to prove Theorem \ref{mainpara}. In Section \ref{sec4},
we apply Theorem \ref{main} to discuss Liouville-type theorems for some parabolic
cases of the equation \eqref{Yameq}, especially for Theorems \ref{app1} and \ref{app3}.
In Section \ref{sec5}, we apply Theorems \ref{main} and \ref{mainpara} to study
Liouville-type theorems for various elliptic versions of the equation \eqref{Yameq};
see for example Theorems \ref{ellliou}, \ref{ellap1} and \ref{ellap2}. In particular,
using these results, we study some Yamabe-type problems of complete manifolds and
smooth metric measure spaces; see Theorems \ref{app4}, \ref{generalcase}, \ref{app5}
and Corollary \ref{app6}.

\vspace{.1in}

\textbf{Acknowledgement}
The author thanks Professor Jeffrey S. Case for helpful discussions.
The author also thanks the referee for making valuable comments and suggestions
and pointing out many errors which helped to improve the exposition of the paper.
This work is supported by the NSFC (11671141) and the Natural Science Foundation
of Shanghai (17ZR1412800).

%%%%%%%%%%%%%%%%%%%%%%%%%%%%%%%%%%%%%%%%%%%%%%%%%%%%%%%%%%%%%%%%%%5

\section{Elliptic gradient estimate}\label{sec2}
In this section, we first prove a lemma, which is a generalization
of \cite{[Sou-Zh],[Wu15]}. Then we apply this lemma and the maximum
principle to prove Theorem \ref{main}.

Let $(M,g,e^{-f}dv)$ be an $n$-dimensional complete smooth metric measure space.
For any point $x_0\in M$ and $R>0$, assume that $0<u(x,t)\leq D$ for some
constant $D$ is a smooth solution to the equation \eqref{Yameq} in $Q_{R,T}$,
where
\[
Q_{R,T}:\equiv B(x_0,R)\times[t_0-T,t_0]\subset M\times(-\infty,\infty),\quad
t_0\in \mathbb{R},\,\,T>0.
\]
Introduce a auxiliary function
\[
h(x,t):=\ln\frac{u}{D}
\]
in $Q_{R,T}$. Then $h\leq 0$ and $h$ satisfies
\begin{equation}\label{lemequ}
\left(\Delta_f-\frac{\partial}{\partial t}\right)h+|\nabla h|^2+p(x,t)(De^h)^{\alpha-1}+q(x,t)(De^h)^{\beta-1}+\mu(x,t)=0.
\end{equation}
Using \eqref{lemequ} we have the following lemma, which will play an
significant part in the proof of Theorem \ref{main}.

\begin{lemma}\label{Lem2.1}
Let $(M,g,e^{-f}dv)$ be a complete smooth metric measure space. Assume that
$\mathrm{Ric}_f\geq-(n-1)K$ for some constant $K\geq0$ in $B(x_0,R)$, where
$x_0\in M$ and $R>0$. Let $h(x,t)$ is a nonpositive smooth function defined
in $Q_{R,T}$ satisfying \eqref{lemequ}. Then the function
\begin{equation}\label{lemmw2}
\omega:=\left|\nabla\ln(1-h)\right|^2=\frac{|\nabla h|^2}{(1-h)^2}
\end{equation}
satisfies
\begin{equation*}
\begin{aligned}
\frac 12\left(\Delta_f-\frac{\partial}{\partial t}\right)\omega
&\geq\frac{h}{1-h}\left\langle \nabla
h,\nabla\omega\right\rangle+(1-h)\omega^2-(n-1)K\omega\\
&\quad-\left(\alpha-1+\frac{1}{1-h}\right)p(De^h)^{\alpha-1}\omega
-\frac{(De^h)^{\alpha{-}1}}{(1-h)^2}\langle\nabla p,\nabla h\rangle\\
&\quad-\left(\beta-1+\frac{1}{1-h}\right)q(De^h)^{\beta-1}\omega
-\frac{(De^h)^{\beta-1}}{(1-h)^2}\langle\nabla q,\nabla h\rangle\\
&\quad-\frac{\mu}{1-h}\omega-\frac{1}{(1-h)^2}\langle\nabla \mu,\nabla h\rangle
\end{aligned}
\end{equation*}
for all $(x,t)$ in $Q_{R,T}$.
\end{lemma}
\begin{proof}
The proof is similar to that of Lemma 2.1 in \cite{[Wu15]}, but is included
for completeness. We shall apply local coordinates to conveniently compute these complicated
evolution equations. Let $e_1, e_2,..., e_n$ be a local orthonormal frame field at a point
$x\in M^n$ and we adopt the notation that subscripts in $i$, $j$, and $k$, with
$1\leq i, j, k\leq n$, mean covariant differentiations in the $e_i$, $e_j$ and $e_k$,
directions respectively. We denote $h_i:=\nabla_i h$, $h_{ii}=\nabla_i\nabla_ih=\Delta h$
and $h_{ijj}:=\nabla_j\nabla_j\nabla_i h$, etc.

By the definition of $\omega$ in \eqref{lemmw2}, we compute that
\begin{equation}\label{lempro1}
\omega_j=\frac{2h_i h_{ij}}{(1-h)^2}+\frac{2h^2_i
h_j}{(1-h)^3},
\end{equation}
\[
\langle\nabla f,\nabla\omega\rangle=\frac{2h_{ij}h_if_j}{(1-h)^2}
+\frac{2h^2_ih_jf_j}{(1-h)^3},
\]
and
\[
\Delta\omega=\frac{2|h_{ij}|^2}{(1-h)^2}+\frac{2h_i
h_{ijj}}{(1-h)^2}+\frac{8h_ih_j
h_{ij}}{(1-h)^3}+\frac{2h^2_i
h_{jj}}{(1-h)^3}+\frac{6h^2_i h^2_j}{(1-h)^4}.
\]
Hence,
\begin{equation*}
\begin{aligned}
\Delta_f\,\omega&=\Delta\omega-\langle\nabla f,\nabla\omega\rangle\\
&=\frac{2|h_{ij}|^2}{(1-h)^2}+\frac{2h_i
h_{ijj}}{(1-h)^2}+\frac{8h_ih_j h_{ij}}{(1-h)^3}
+\frac{2h^2_i h_{jj}}{(1-h)^3}+\frac{6h^4_i}{(1-h)^4}\\
&\quad-\frac{2h_{ij}h_if_j}{(1-h)^2}
-\frac{2h^2_ih_jf_j}{(1-h)^3}.
\end{aligned}
\end{equation*}
Using the Ricci identity $h_{ijj}=h_{jji}+\mathrm{R}_{ij}h_j$, the above
inequality becomes
\begin{equation}
\begin{aligned}\label{bak-em}
\Delta_f\,\omega
&=\frac{2|h_{ij}|^2}{(1-h)^2}+\frac{2h_i
(\Delta_f h)_i}{(1-h)^2}+\frac{2(\mathrm{R}_{ij}+f_{ij})h_ih_j}{(1-h)^2}
+\frac{8h_ih_jh_{ij}}{(1-h)^3}\\
&\quad+\frac{6h^4_i}{(1-h)^4}+\frac{2h^2_i\cdot\Delta_fh}{(1-h)^3}.
\end{aligned}
\end{equation}
From \eqref{lemequ} and \eqref{lemmw2}, we obtain
\begin{equation}
\begin{aligned}\label{drivatt}
\frac{\partial\omega}{\partial t}&=\frac{2\nabla_ih\cdot\nabla_i\left(\Delta_fh+|\nabla
h|^2+p(De^h)^{\alpha-1}+q(De^h)^{\beta-1}+\mu\right)}{(1-h)^2}\\
&\quad+\frac{2|\nabla h|^2\left(\Delta_fh+|\nabla
h|^2+p(De^h)^{\alpha-1}+q(De^h)^{\beta-1}+\mu\right)}{(1-h)^3}\\
&=\frac{2\nabla h \nabla \Delta_fh}{(1-h)^2}+\frac{4h_ih_jh_{ij}}{(1-h)^2}
+\frac{2h^2_i\Delta_f h}{(1-h)^3}+\frac{2|\nabla h|^4}{(1-h)^3}\\
&\quad+2\left(\alpha-1+\frac{1}{1-h}\right)p(De^h)^{\alpha-1}\omega
+\frac{2p_ih_i(De^h)^{\alpha-1}}{(1-h)^2}\\
&\quad+2\left(\beta-1+\frac{1}{1-h}\right)q(De^h)^{\beta-1}\omega
+\frac{2q_ih_i(De^h)^{\beta-1}}{(1-h)^2}\\
&\quad+\frac{2\mu}{1-h}\omega+\frac{2\mu_ih_i}{(1-h)^2}.
\end{aligned}
\end{equation}
Combining \eqref{bak-em} and \eqref{drivatt}, we get
\begin{equation*}
\begin{aligned}
\frac 12\left(\Delta_f-\frac{\partial}{\partial
t}\right)\omega&=\frac{|h_{ij}|^2}{(1-h)^2}
+\frac{(\mathrm{R}_{ij}+f_{ij})h_ih_j}{(1-h)^2}+\frac{4h_ih_jh_{ij}}{(1-h)^3}\\
&\quad+\frac{3h^4_i}{(1-h)^4}-\frac{2h_ih_jh_{ij}}{(1-h)^2}-\frac{h^4_i}{(1-h)^3}\\
&\quad-\left(\alpha-1+\frac{1}{1-h}\right)p(De^h)^{\alpha-1}\omega
-\frac{p_ih_i(De^h)^{\alpha-1}}{(1-h)^2}\\
&\quad-\left(\beta-1+\frac{1}{1-h}\right)q(De^h)^{\beta-1}\omega
-\frac{q_ih_i(De^h)^{\beta-1}}{(1-h)^2}\\
&\quad-\frac{\mu}{1-h}\omega-\frac{\mu_ih_i}{(1-h)^2}.
\end{aligned}
\end{equation*}
Since $\mathrm{Ric}_f\geq-(n-1)K$ for some constant $K\geq0$, we have
\[
(\mathrm{R}_{ij}+f_{ij})h_ih_j\geq-(n-1)Kh_i^2.
\]
Since $1-h\geq 1$, we also have
\[
\frac{|h_{ij}|^2}{(1-h)^2}+\frac{2h_ih_j
h_{ij}}{(1-h)^3}+\frac{h^4_i}{(1-h)^4}\geq 0.
\]
Using these two inequalities, the above equation can be simplified as
\begin{equation}
\begin{aligned}\label{dgele}
\frac 12\left(\Delta_f-\frac{\partial}{\partial
t}\right)\omega&\geq-\frac{(n-1)Kh^2_i}{(1-h)^2}+\frac{2h_ih_j
h_{ij}}{(1-h)^3}+\frac{2h^4_i}{(1-h)^4}\\
&\quad-\frac{2h_ih_jh_{ij}}{(1-h)^2}-\frac{h^4_i}{(1-h)^3}\\
&\quad-\left(\alpha-1+\frac{1}{1-h}\right)p(De^h)^{\alpha-1}\omega
-\frac{p_ih_i(De^h)^{\alpha-1}}{(1-h)^2}\\
&\quad-\left(\beta-1+\frac{1}{1-h}\right)q(De^h)^{\beta-1}\omega
-\frac{q_ih_i(De^h)^{\beta-1}}{(1-h)^2}\\
&\quad-\frac{\mu}{1-h}\omega-\frac{\mu_ih_i}{(1-h)^2}.
\end{aligned}
\end{equation}
From \eqref{lempro1}, we know that
\[
\omega_jh_j=\frac{2h_ih_j h_{ij}}{(1-h)^2}+\frac{2h^4_i}{(1-h)^3}.
\]
Using this formula, \eqref{dgele} can be rewritten as
\begin{equation*}
\begin{aligned}
\frac 12\left(\Delta_f-\frac{\partial}{\partial
t}\right)\omega&\geq-\frac{(n-1)Kh_i^2}{(1-h)^2}+\frac{h}{1-h}\omega_jh_j
+\frac{h^4_i}{(1-h)^3}\\
&\quad-\left(\alpha-1+\frac{1}{1-h}\right)p(De^h)^{\alpha-1}\omega
-\frac{p_ih_i(De^h)^{\alpha-1}}{(1-h)^2}\\
&\quad-\left(\beta-1+\frac{1}{1-h}\right)q(De^h)^{\beta-1}\omega
-\frac{q_ih_i(De^h)^{\beta-1}}{(1-h)^2}\\
&\quad-\frac{\mu}{1-h}\omega-\frac{\mu_ih_i}{(1-h)^2}.
\end{aligned}
\end{equation*}
By the definition of $\omega$, the desired inequality immediately follows.
\end{proof}

In the rest of this section, we will apply Lemma \ref{Lem2.1} and the
localized technique of Souplet-Zhang \cite{[Sou-Zh]} and the author
\cite{[Wu15]} to give an elliptic-type gradient estimate for positive
smooth solutions to the equation \eqref{Yameq}.

We first introduce a useful space-time cut-off function originated by
Li-Yau \cite{[Li-Yau]} (see also \cite{[Sou-Zh]} and \cite{[Wu15]})
as follows.

\begin{lemma}\label{cutoff}
Fix $t_0\in \mathbb{R}$ and $T>0$. For any given $\tau\in(t_0-T,t_0]$,
there exists a smooth function
$\bar\psi:[0,\infty)\times[t_0-T,t_0]\to\mathbb R$ satisfying
following propositions:
\begin{enumerate}
\item
$0\le\overline{\psi}(r,t)\le 1$ in $[0,R]\times[t_0-T,t_0]$, and it is supported
in a open subset of $[0,R]\times[t_0-T,t_0]$.
\item
$\overline{\psi}(r,t)=1$ and $\partial_r\overline{\psi}(r,t)=0$
in $[0,R/2]\times[\tau,t_0]$ and $[0,R/2]\times[t_0-T,t_0]$, respectively.
\item
$|\partial_t\overline{\psi}|\leq\frac{C}{\tau-(t_0-T)}{\overline{\psi}}^{\frac12}$
in $[0,\infty)\times[t_0-T,t_0]$ for some $C>0$, and $\overline{\psi}(r,t_0-T)=0$
for all $r\in[0,\infty)$.
\item
$-\frac{C_\epsilon}{R}{\overline{\psi}}^\epsilon\leq\partial_r\overline{\psi}\leq 0$
and $|\partial_r^2\overline{\psi}|\leq\frac{C_\epsilon}{R^2}{\overline{\psi}}^\epsilon$
in $[0,\infty)\times[t_0-T,t_0]$ for every $\epsilon\in(0,1)$ with some constant
$C_\epsilon$ depending on $\epsilon$.
\end{enumerate}
\end{lemma}

Then we apply Lemma \ref{Lem2.1} and Lemma \ref{cutoff} to prove Theorem \ref{main}
via the maximum principle in a local space-time supported set. The proof mainly
follows the arguments of \cite{[BCP]} and \cite{[Wu15]}, which is a little different
from \cite{[Sou-Zh]}.

\begin{proof}[Proof of Theorem \ref{main}]
Pick any number $\tau\in(t_0-T,t_0]$ and choose a cutoff function $\bar\psi(r,t)$
satisfying the conditions of Lemma \ref{cutoff}. Briefly, we will show that
the inequalities in Theorem \ref{main} hold at the point $(x,\tau)$ for all
$x\in M$ such that $d(x,x_0)<R/2$. Since $\tau$ is arbitrary, the assertion of
theorem will immediately follow. In the following we provide a detailed description.

Let $\psi:M\times[t_0-T,t_0]\to\mathbb R$ be the cutoff function
$\psi=\overline{\psi}(d(x,x_0),t)\equiv\psi(r,t)$.
Then $\psi(x,t)$ could be viewed as smooth cut-off function supported in $Q_{R,T}$.
Our strategy is to estimate $(\Delta_f-\frac{\partial}{\partial t})(\psi\omega)$
and carefully analyze the result at a space-time point where the function
$\psi\omega$ attains its maximum.

We apply Lemma \ref{Lem2.1} to conclude that
\begin{equation}
\begin{aligned}\label{lemdx3}
\frac 12\left(\Delta_f-\frac{\partial}{\partial t}\right)&(\psi\omega)-\left(\frac{h}{1-h}\nabla
h+\frac{\nabla\psi}{\psi}\right)\cdot\nabla(\psi\omega)\\
&\geq\psi(1-h)\omega^2-\left(\frac{h}{1-h}\nabla h\cdot\nabla\psi\right)\omega
-\frac{|\nabla\psi|^2}{\psi}\omega\\
&\quad+\frac 12(\Delta_f\psi)\omega-\frac 12\psi_t\omega-(n-1)K\psi\omega\\
&\quad-\left(\alpha-1+\frac{1}{1-h}\right)p(De^h)^{\alpha-1}\psi\omega
-\frac{\psi p_ih_i(De^h)^{\alpha-1}}{(1-h)^2}\\
&\quad-\left(\beta-1+\frac{1}{1-h}\right)q(De^h)^{\beta-1}\psi\omega
-\frac{\psi q_ih_i(De^h)^{\beta-1}}{(1-h)^2}\\
&\quad-\frac{\mu}{1-h}\psi\omega-\frac{\psi\mu_ih_i}{(1-h)^2}.
\end{aligned}
\end{equation}
Now let $(x_1,t_1)$ be a maximum space-time point for $\psi\omega$ in the closed set
\[
\left\{(x,t)\in M\times[t_0-T,\tau]\,|d(x,x_0)\leq R\right\}.
\]
We may assume $(\psi\omega)(x_1,t_1)>0$; otherwise, $\omega(x,\tau)\leq0$ and
the conclusion naturally holds at $(x,\tau)$ whenever $d(x, x_0)<\frac R2$.
Notice that $t_1\neq t_0-T$, since we assume $(\psi\omega)(x_1,t_1)>0$. We may
also assume that $\psi(x,t)$ is smooth at $(x_1,t_1)$ due to the standard
Calabi argument \cite{[Cala]}. Since $(x_1,t_1)$ is a maximum space-time
point, at this point we have
\[
\Delta_f(\psi\omega)\leq0,\quad(\psi\omega)_t\geq0
\quad \mathrm{and}\quad\nabla(\psi\omega)=0.
\]
Using the above  estimates at $(x_1,t_1)$, \eqref{lemdx3} can be simplified as
\begin{equation}\label{lefor}
\begin{aligned}
\psi(1-h)\omega^2\leq&\left(\frac{h}{1-h}\nabla h\cdot\nabla\psi
+\frac{|\nabla\psi|^2}{\psi}\right)\omega\\
&-\frac 12(\Delta_f\psi)\omega+\frac 12\psi_t\omega+(n-1)K\psi\omega\\
&+\left(\alpha-1+\frac{1}{1-h}\right)pu^{\alpha-1}\psi\omega
+\frac{\psi p_ih_iu^{\alpha-1}}{(1{-}h)^2}\\
&+\left(\beta-1+\frac{1}{1-h}\right)qu^{\beta-1}\psi\omega
+\frac{\psi q_ih_iu^{\beta-1}}{(1{-}h)^2}\\
&+\frac{\mu}{1-h}\psi\omega+\frac{\psi\mu_ih_i}{(1-h)^2}
\end{aligned}
\end{equation}
at $(x_1,t_1)$,  where in the above estimates we have used the fact that $u=D e^h$.

\

In the rest, we will use \eqref{lefor} at the maximum space-time point
$(x_1,t_1)$ to give the desired gradient estimate in Theorem \ref{main}. We will
achieve it by two steps.

\vspace{0.5em}

\textbf{Case One}: We assume the maximum space-point $x_1\not\in B(x_0,1)$. Recall
that, $\mathrm{Ric}_f\geq-(n-1)K$ and $r(x_1,x_0)\geq 1$ in $B(x_0,R)$, $R\geq 2$. Hence by
the $f$-Laplacian comparison theorem (Theorem 3.1 in \cite{[WW]}), we have
\begin{equation}\label{gencomp}
\Delta_f\,r(x_1)\leq\sigma+(n-1)K(R-1),
\end{equation}
where $\sigma:=\max_{\{x|d(x,x_0)=1\}}\Delta_f\,r(x)$, which will be used later.
Below we will carefully estimate upper bounds for each term on the right hand
side (RHS) of \eqref{lefor}, similar to the arguments of Souplet-Zhang \cite{[Sou-Zh]}
and the author \cite{[Wu15]}. This will lead us to give the desired result. We remark
that the Young's inequality will be repeatedly used in the following estimates. Below
we let $c$ denote a constant depending only on $n$ whose value may change from line to line.

First, we estimate the first term on the RHS of \eqref{lefor}:
\begin{equation}
\begin{aligned}\label{term1}
\left(\frac{h}{1-h}\nabla h\cdot\nabla\psi\right)\omega
&\leq|h|\cdot|\nabla\psi|\cdot\omega^{3/2}\\
&=\left[\psi(1-h)\omega^2\right]^{3/4}
\cdot\frac{|h|\cdot|\nabla\psi|}{[\psi(1-h)]^{3/4}}\\
&\leq\frac 13\psi(1-h)\omega^2+c
\frac{(h|\nabla\psi|)^4}{[\psi(1-h)]^3}\\
&\leq\frac 13\psi(1-h)\omega^2+\frac{ch^4}{R^4(1-h)^3}.
\end{aligned}
\end{equation}

For the second term on the RHS of (\ref{lefor}), we have
\begin{equation}
\begin{aligned}\label{term2}
\frac{|\nabla\psi|^2}{\psi}\omega
&=\psi^{1/2}\omega\cdot\frac{|\nabla\psi|^2}{\psi^{3/2}}\\
&\leq\frac{1}{18}\psi\omega^2+c
\left(\frac{|\nabla\psi|^2}{\psi^{3/2}}\right)^2\\
&\leq\frac{1}{18}\psi\omega^2+\frac{c}{R^4}.
\end{aligned}
\end{equation}

For the third term on the RHS of \eqref{lefor}, since $\psi$
is a radial function, then at $(x_1,t_1)$, using \eqref{gencomp} we have
\begin{equation}
\begin{aligned}\label{term3}
-\frac 12(\Delta_f\psi)\omega&=-\frac 12\left[(\partial_r\psi)\Delta_fr+(\partial^2_r\psi)
|\nabla r|^2\right]\omega\\
&\leq-\frac 12\left[\partial_r\psi\left(\sigma+(n-1)K(R-1)\right)+\partial^2_r\psi\right]\omega\\
&\leq\left[|\partial^2_r\psi|+\left(\sigma^++(n-1)K(R-1)\right)|\partial_r\psi|\right]\omega\\
&=\psi^{1/2}\omega\frac{|\partial^2_r\psi|}{\psi^{1/2}}
+\psi^{1/2}\omega[\sigma^++(n-1)K(R-1)]\frac{|\partial_r\psi|}{\psi^{1/2}}\\
&\leq\frac{1}{18}\psi\omega^2+c\frac{|\partial^2_r\psi|^2}{\psi}
+c\frac{(\sigma^+)^2|\partial_r\psi|^2}{\psi}
+c\frac{K^2(R-1)^2|\partial_r\psi|^2}{\psi}\\
&\leq\frac{1}{18}\psi\omega^2+\frac{c}{R^4}+\frac{c(\sigma^+)^2}{R^2}+cK^2,
\end{aligned}
\end{equation}
where $\sigma^+:=\max\{\sigma,0\}$, and in the last inequality we have used proposition (4)
in Lemma \ref{cutoff}.

For the fourth term on the RHS of \eqref{lefor}, we have
\begin{equation}
\begin{aligned}\label{term4}
\frac 12|\psi_t|\omega&=\frac 12\psi^{1/2}\omega\frac{|\psi_t|}{\psi^{1/2}}\\
&\leq\frac{1}{18}\left(\psi^{1/2}\omega\right)^2+c
\left(\frac{|\psi_t|}{\psi^{1/2}}\right)^2\\
&\leq\frac{1}{18}\psi\omega^2+\frac{c}{(\tau-t_0+T)^2}.
\end{aligned}
\end{equation}

For the fifth term on the RHS of \eqref{lefor}, we have
\begin{equation}
\begin{aligned}\label{term5}
(n-1)K\psi\omega&=(n-1)\psi^{1/2}\omega\cdot\psi^{1/2}K\\
&\leq\frac{1}{18}\psi\omega^2+cK^2.
\end{aligned}
\end{equation}

For the sixth term on the RHS of \eqref{lefor}, we easily get
\begin{equation*}
\begin{aligned}
\Big(\alpha-1+\frac{1}{1-h}\Big)p&\leq(\alpha-1)p+\frac{p^+}{1-h}\\
&\leq[(\alpha-1)p]^++p^+,
\end{aligned}
\end{equation*}
where in the above inequality we used the fact $\frac{1}{1-h}>1$ due to
$h\leq 0$. Hence we have
\begin{equation}
\begin{aligned}\label{term8a}
\Big(\alpha-1+\frac{1}{1-h}\Big)pu^{\alpha-1}\psi\omega
&\leq\Big([(\alpha-1)p]^++p^+\Big)\,u^{\alpha-1}\psi\omega\\
&\leq\frac{1}{18}\psi\omega^2+c\Big([(\alpha-1)p]^++p^+\Big)^2
\sup_{Q_{R,T}}\{u^{2(\alpha-1)}\},
\end{aligned}
\end{equation}
where in the last inequality, we have used the fact $\psi\leq 1$.

For the seventh term on the RHS of \eqref{lefor}, since $h<0$,
we have the following estimate
\begin{equation}
\begin{aligned}\label{term9a}
\frac{u^{\alpha-1}}{(1-h)^2}\psi p_ih_i
&\leq\frac{u^{\alpha-1}}{(1-h)^2}\psi|p_i|\cdot|h_i|\\
&\leq u^{\alpha-1}\psi\omega^{1/2}|\nabla p(x_1,t_1)|\\
&\leq\frac{1}{18}(\psi^{1/4}\omega^{1/2})^4+c\left(\psi^{3/4}u^{\alpha-1}|\nabla p(x_1,t_1)|\right)^{\frac 43}\\
&\leq\frac{1}{18}\psi\omega^2+c\sup_{Q_{R,T}}|\nabla p|^{\frac 43}\sup_{Q_{R,T}}\{u^{\frac 43(\alpha-1)}\}.
\end{aligned}
\end{equation}

For the eighth and ninth terms on the RHS of \eqref{lefor}, the estimates are
very similar to the sixth and seventh terms. We summarize these estimates without
providing the detailed proof.
\begin{equation}\label{term8ak}
\left(\beta-1+\frac{1}{1-h}\right)qu^{\beta-1}\psi\omega\leq
\frac{1}{18}\psi\omega^2+c\Big([(\beta-1)q]^++q^+\Big)^2
\sup_{Q_{R,T}}\{u^{2(\beta-1)}\}
\end{equation}
and
\begin{equation}\label{term9ak}
\frac{u^{\beta-1}}{(1-h)^2}\psi q_ih_i\leq
\frac{1}{18}\psi\omega^2+c\sup_{Q_{R,T}}|\nabla q|^{\frac 43}\sup_{Q_{R,T}}\{u^{\frac 43(\beta-1)}\}.
\end{equation}

For the tenth term on the RHS of \eqref{lefor}, similar to \eqref{term8a},
we have the following estimate
\begin{equation}\label{term10}
\frac{\mu}{1-h}\psi\omega\leq\frac{1}{18}\psi\omega^2+c(\mu^+)^2,
\end{equation}
where $\mu^+:=\sup_{(x,t)\in Q_{R,T}}\{\mu^+(x,t),0\}$ and $\mu^+(x,t)=\max\{\mu(x,t),0\}$.
For the eleventh term on the RHS of \eqref{lefor}, similar to \eqref{term9a}, we have the estimate
\begin{equation}\label{term11}
\frac{\psi\mu_ih_i}{(1-h)^2}\leq\frac{1}{18}\psi\omega^2+c\sup_{Q_{R,T}}|\nabla\mu|^{\frac 43}.
\end{equation}

\

In the following, we will apply the above estimates to prove the theorem.
Substituting \eqref{term1}-\eqref{term11} into the RHS of \eqref{lefor},
at $(x_1,t_1)$, we have that
\begin{equation*}
\begin{aligned}
\psi(1-h)\omega^2&\leq \frac 13\psi(1-h)\omega^2+\frac{ch^4}{R^4(1-h)^3}
+\frac{10}{18}\psi\omega^2\\
&\quad+\frac{c}{R^4}+\frac{c(\sigma^+)^2}{R^2}+\frac{c}{(\tau-t_0+T)^2}+cK^2
+c(\mu^+)^2+c\sup_{Q_{R,T}}|\nabla\mu|^{\frac 43}\\
&\quad+c\Big([(\alpha-1)p]^++p^+\Big)^2\sup_{Q_{R,T}}\{u^{2(\alpha-1)}\}
+c\sup_{Q_{R,T}}|\nabla p|^{\frac 43}\sup_{Q_{R,T}}\{u^{\frac 43(\alpha-1)}\}\\
&\quad+c\Big([(\beta-1)q]^++q^+\Big)^2\sup_{Q_{R,T}}\{u^{2(\beta-1)}\}
+c\sup_{Q_{R,T}}|\nabla q|^{\frac 43}\sup_{Q_{R,T}}\{u^{\frac 43(\beta-1)}\}.
\end{aligned}
\end{equation*}
Since $1-h\geq1$, the above estimate implies
\begin{equation*}
\begin{aligned}
\psi\omega^2&\leq\frac{ch^4}{R^4(1-h)^4}
+\frac{c}{R^4}+\frac{c(\sigma^+)^2}{R^2}+\frac{c}{(\tau-t_0+T)^2}+cK^2
+c(\mu^+)^2+c\sup_{Q_{R,T}}|\nabla\mu|^{\frac 43}\\
&\quad+c\Big([(\alpha-1)p]^++p^+\Big)^2\sup_{Q_{R,T}}\{u^{2(\alpha-1)}\}
+c\sup_{Q_{R,T}}|\nabla p|^{\frac 43}\sup_{Q_{R,T}}\{u^{\frac 43(\alpha-1)}\}\\
&\quad+c\Big([(\beta-1)q]^++q^+\Big)^2\sup_{Q_{R,T}}\{u^{2(\beta-1)}\}
+c\sup_{Q_{R,T}}|\nabla q|^{\frac 43}\sup_{Q_{R,T}}\{u^{\frac 43(\beta-1)}\}
\end{aligned}
\end{equation*}
at $(x_1,t_1)$. Moreover, since $\frac{h^4}{(1-h)^4}\leq 1$,
the above inequality implies that
\begin{equation*}
\begin{aligned}
(\psi^2\omega^2)(x_1,t_1)&\leq(\psi\omega^2)(x_1,t_1)\\
&\leq\frac{c}{R^4}+\frac{c(\sigma^+)^2}{R^2}+\frac{c}{(\tau-t_0+T)^2}+cK^2
+c(\mu^+)^2+c\sup_{Q_{R,T}}|\nabla\mu|^{\frac 43}\\
&\quad+c\Big([(\alpha-1)p]^++p^+\Big)^2\sup_{Q_{R,T}}\{u^{2(\alpha-1)}\}
+c\sup_{Q_{R,T}}|\nabla p|^{\frac 43}\sup_{Q_{R,T}}\{u^{\frac 43(\alpha-1)}\}\\
&\quad+c\Big([(\beta-1)q]^++q^+\Big)^2\sup_{Q_{R,T}}\{u^{2(\beta-1)}\}
+c\sup_{Q_{R,T}}|\nabla q|^{\frac 43}\sup_{Q_{R,T}}\{u^{\frac 43(\beta-1)}\}.
\end{aligned}
\end{equation*}
Since $\psi(x,\tau)=1$ when $d(x,x_0)<R/2$ by the proposition (2)
in Lemma \ref{cutoff}, from the above estimate, we in fact get
\begin{equation*}
\begin{aligned}
\omega(x,\tau)&=(\psi\omega)(x,\tau)\\
&\leq(\psi\omega)(x_1,t_1)\\
&\leq\frac{c}{R^2}+\frac{c\sigma^+}{R}+\frac{c}{\tau-t_0+T}+cK
+c\mu^++c\sup_{Q_{R,T}}|\nabla\mu|^{\frac 23}\\
&\quad+c\Big([(\alpha-1)p]^++p^+\Big)\sup_{Q_{R,T}}\{u^{\alpha-1}\}
+c\sup_{Q_{R,T}}|\nabla p|^{\frac 23}\sup_{Q_{R,T}}\{u^{\frac 23(\alpha-1)}\}\\
&\quad+c\Big([(\beta-1)q]^++q^+\Big)\sup_{Q_{R,T}}\{u^{\beta-1}\}
+c\sup_{Q_{R,T}}|\nabla q|^{\frac 23}\sup_{Q_{R,T}}\{u^{\frac 23(\beta-1)}\}
\end{aligned}
\end{equation*}
for all $x\in M$ such that $d(x,x_0)<R/2$. By the definition of
$w(x,\tau)$ and the fact that $\tau\in(t_0-T,t_0]$ was chosen
arbitrarily, we get the estimate
\begin{equation*}
\begin{aligned}
\frac{|\nabla h|}{(1-h)}(x,t)&\leq
\frac cR+c\sqrt{\frac{\sigma^+}{R}}+\frac{c}{\sqrt{t-t_0+T}}+c\sqrt{K}
+c\sqrt{\mu^+}+c\sup_{Q_{R,T}}|\nabla\mu|^{\frac 13}\\
&\quad+c\sqrt{[(\alpha-1)p]^++p^+}\cdot\sup_{Q_{R,T}}\{u^{\frac{\alpha-1}{2}}\}
+c\sup_{Q_{R,T}}|\nabla p|^{\frac 13}\sup_{Q_{R,T}}\{u^{\frac{\alpha-1}{3}}\}\\
&\quad+c\sqrt{[(\beta-1)q]^++q^+}\cdot\sup_{Q_{R,T}}\{u^{\frac{\beta-1}{2}}\}
+c\sup_{Q_{R,T}}|\nabla q|^{\frac 13}\sup_{Q_{R,T}}\{u^{\frac{\beta-1}{3}}\}
\end{aligned}
\end{equation*}
for all $(x,t)\in Q_{R/2,T}$ with $t\neq t_0-T$. Since $h=\ln(u/D)$,
substituting this into the above estimate completes the proof of
theorem when $x_1\not\in B(x_0,1)\subset B(x_0,R)$, where $R\geq 2$.

\vspace{.1in}

\textbf{Case Two}: We assume the maximum space-point $x_1\in B(x_0,1)$.
In this case, $\psi$ is constant in space direction in $B(x_0,R/2)$ by our assumption,
where $R\geq2$. So by \eqref{lefor}, we have
\begin{equation*}
\begin{aligned}
\psi\omega^2&\leq\frac 12\psi_t\omega+(n-1)K\psi\omega
+\frac{\mu}{1-h}\psi\omega+\frac{\psi\mu_ih_i}{(1-h)^2}\\
&\quad+\left(\alpha-1+\frac{1}{1-h}\right)p(De^h)^{\alpha-1}\psi\omega
+\frac{\psi p_ih_i(De^h)^{\alpha-1}}{(1-h)^2}\\
&\quad+\left(\beta-1+\frac{1}{1-h}\right)q(De^h)^{\beta-1}\psi\omega
+\frac{\psi q_ih_i(De^h)^{\beta-1}}{(1-h)^2}
\end{aligned}
\end{equation*}
at $(x_1,t_1)$, where we have used $1-h\geq 1$ on the left hand side of the above inequality.
By \eqref{term4}--\eqref{term11}, the above inequality can be estimated by
\begin{equation*}
\begin{aligned}
\psi\omega^2&\leq\frac{8}{18}\psi\omega^2+\frac{c}{(\tau-t_0+T)^2}+cK^2
+c(\mu^+)^2+c\sup_{Q_{R,T}}|\nabla\mu|^{\frac 43}\\
&\quad+c\Big([(\alpha-1)p]^++p^+\Big)^2\sup_{Q_{R,T}}\{u^{2(\alpha-1)}\}
+c\sup_{Q_{R,T}}|\nabla p|^{\frac 43}\sup_{Q_{R,T}}\{u^{\frac 43(\alpha-1)}\}\\
&\quad+c\Big([(\beta-1)q]^++q^+\Big)^2\sup_{Q_{R,T}}\{u^{2(\beta-1)}\}
+c\sup_{Q_{R,T}}|\nabla q|^{\frac 43}\sup_{Q_{R,T}}\{u^{\frac 43(\beta-1)}\}
\end{aligned}
\end{equation*}
at $(x_1,t_1)$. Since $\psi(x_1,t_1)=1$, the above inequality can be written as
\begin{equation*}
\begin{aligned}
\omega^2(x_1,t_1)&\leq\frac{c}{(\tau-t_0+T)^2}+cK^2+c(\mu^+)^2+c\sup_{Q_{R,T}}|\nabla\mu|^{\frac 43}\\
&\quad+c\Big([(\alpha-1)p]^++p^+\Big)^2\sup_{Q_{R,T}}\{u^{2(\alpha-1)}\}
+c\sup_{Q_{R,T}}|\nabla p|^{\frac 43}\sup_{Q_{R,T}}\{u^{\frac 43(\alpha-1)}\}\\
&\quad+c\Big([(\beta-1)q]^++q^+\Big)^2\sup_{Q_{R,T}}\{u^{2(\beta-1)}\}
+c\sup_{Q_{R,T}}|\nabla q|^{\frac 43}\sup_{Q_{R,T}}\{u^{\frac 43(\beta-1)}\}.
\end{aligned}
\end{equation*}
Since $\psi(x,\tau)=1$ when $d(x,x_0)<R/2$ by the proposition
(2) in Lemma \ref{cutoff}, the above estimate indeed gives that
\begin{equation*}
\begin{aligned}
\omega(x,\tau)&=(\psi\omega)(x,\tau)\\
&\leq(\psi\omega)(x_1,t_1)\\
&\leq\omega(x_1,t_1)\\
&\leq\frac{c}{\tau-t_0+T}+cK+c\mu^++c\sup_{Q_{R,T}}|\nabla\mu|^{\frac 23}\\
&\quad+c\Big([(\alpha-1)p]^++p^+\Big)\sup_{Q_{R,T}}\{u^{\alpha-1}\}
+c\sup_{Q_{R,T}}|\nabla p|^{\frac 23}\sup_{Q_{R,T}}\{u^{\frac 23(\alpha-1)}\}\\
&\quad+c\Big([(\beta-1)q]^++q^+\Big)\sup_{Q_{R,T}}\{u^{\beta-1}\}
+c\sup_{Q_{R,T}}|\nabla q|^{\frac 23}\sup_{Q_{R,T}}\{u^{\frac 23(\beta-1)}\}
\end{aligned}
\end{equation*}
for all $x\in M$ such that $d(x,x_0)<R/2$. By the definition of
$w(x,\tau)$ and the fact that $\tau\in(t_0-T,t_0]$ was chosen
arbitrarily, we in fact prove that the estimate in theorem still holds
when $x_1\in B(x_0,1)$.
\end{proof}

%%%%%%%%%%%%%%%%%%%%%%%%%%%%%%%%%%%%%%%%%%%%%%%%%%%%%%%%%%%%%%%%%%5

\section{Parabolic gradient estimate}\label{sec3}
In this section, by adapting the arguments of \cite{[Li-Yau],[Wu10]}, we
first give a useful lemma. Then we apply the lemma to prove Theorem \ref{mainpara}
by the maximum principle in a locally supported set of the manifold.

Let $(M,g,e^{-f}dv)$ be an $n$-dimensional complete smooth metric measure space. For
any point $x_0\in M$ and $R>0$, assume that $u(x,t)$ is a positive smooth solution to
the equation \eqref{Yameq} in $H_{2R,T}$, where $H_{2R,T}:=B(x_0,2R)\times[0,T]$, $T>0$.
Introduce a auxiliary function
\[
h(x,t):=\ln u(x,t)
\]
in $H_{2R,T}$. By the equation \eqref{Yameq}, the function $h$ satisfies
\begin{equation}\label{lemequ2l}
\left(\Delta_f-\frac{\partial}{\partial t}\right)h+|\nabla h|^2+p(x,t)(e^h)^{\alpha-1}
+q(x,t)(e^h)^{\beta-1}+\mu(x,t)=0.
\end{equation}
Then we have the following useful lemma, which will be important in the proof of
Theorem \ref{mainpara}.

\begin{lemma}\label{Lem3.1}
Let $(M,g,e^{-f}dv)$ be a complete smooth metric measure space. Assume that
$\mathrm{Ric}^m_f\geq-(m+n-1)K$ for some constant $K\geq0$ in $B(x_0,2R)$, where
$x_0\in M$ and $R>0$. Let $h(x,t)$ be a smooth function in $H_{2R,T}$
satisfying the equation \eqref{lemequ2l}. Then for any $\lambda>1$, the function
\begin{equation}\label{lemmaequ2}
F:=t\left[|\nabla h|^2-\lambda\Big(h_t-p(e^h)^{\alpha-1}-q(e^h)^{\beta-1}-\mu\Big)\right]
\end{equation}
satisfies
\begin{equation*}
\begin{aligned}
\left(\Delta_f-\frac{\partial}{\partial t}\right)F&\geq-\frac Ft
-2\langle\nabla h,\nabla F\rangle-2(m+n-1)Kt|\nabla h|^2\\
&\quad+\frac{2t}{m+n}\left[|\nabla h|^2+p(e^h)^{\alpha-1}+q(e^h)^{\beta-1}+\mu-h_t\right]^2\\
&\quad-(\alpha-1)p(e^h)^{\alpha-1}F+2(\lambda\alpha-1)t(e^h)^{\alpha-1}\langle\nabla h,\nabla p\rangle\\
&\quad+(\alpha-1)(\lambda\alpha-1)tp(e^h)^{\alpha-1}|\nabla h|^2
+\lambda t(e^h)^{\alpha-1}\Delta_fp\\
&\quad-(\beta-1)q(e^h)^{\beta-1}F+2(\lambda\beta-1)t(e^h)^{\beta-1}\langle\nabla h,\nabla q\rangle\\
&\quad+(\beta-1)(\lambda\beta-1)tq(e^h)^{\beta-1}|\nabla h|^2
+\lambda t(e^h)^{\beta-1}\Delta_fq\\
&\quad+2(\lambda-1)t\langle\nabla h,\nabla\mu\rangle+\lambda t\Delta_f\mu
\end{aligned}
\end{equation*}
for all $(x,t)$ in $H_{2R,T}$.
\end{lemma}
\begin{proof}[Proof of Lemma \ref{Lem3.1}]
The proof follows by direct computations.
Using \eqref{lemequ2l} and \eqref{lemmaequ2}, by the definition of $F$, we compute that
\[
\Delta_f F=t\left[\Delta_f|\nabla h|^2-\lambda\Delta_fh_t+\lambda\Delta_f\mathcal{D}\right],
\]
where $\mathcal{D}:=p(e^h)^{\alpha-1}+q(e^h)^{\beta-1}+\mu$.  By the Bochner formula of the
$m$-Bakry-\'Emery Ricci tensor and the assumption $\mathrm{Ric}^m_f\geq-(m+n-1)K$, we have
\[
\Delta_f|\nabla h|^2\geq\frac{2(\Delta_f h)^2}{m+n}+2\langle\nabla h,\nabla\Delta_f h\rangle-2(m+n-1)K|\nabla h|^2.
\]
Hence,
\begin{equation}\label{bake-em2}
\Delta_f F\geq t\left[\frac{2(\Delta_f h)^2}{m+n}+2\langle\nabla h,\nabla\Delta_f h\rangle
-2(m+n-1)K|\nabla h|^2-\lambda\Delta_fh_t+\lambda\Delta_f\mathcal{D}\right].
\end{equation}
Notice that by \eqref{lemequ2l} and \eqref{lemmaequ2}, we have the following equality:
\begin{equation}
\begin{aligned}\label{impform}
-\lambda\Delta_fh_t{+}2\langle\nabla h,\nabla\Delta_f h\rangle&=
\frac{F_t}{t}-\frac{F}{t^2}+2(\lambda{-}1)\nabla h\nabla h_t+2\langle\nabla h,\nabla\Delta_f h\rangle\\
&=\frac{F_t}{t}-\frac{F}{t^2}+2(\lambda{-}1)\nabla h\nabla(\Delta_fh{+}|\nabla h|^2{+}\mathcal{D})
{+}2\langle\nabla h,\nabla\Delta_f h\rangle\\
&=\frac{F_t}{t}-\frac{F}{t^2}-\frac 2t\langle\nabla h,\nabla F\rangle
+2(\lambda-1)\langle\nabla h,\nabla\mathcal{D}\rangle,
\end{aligned}
\end{equation}
where in the last equality we have used the following formulae
\begin{equation}\label{identity}
\Delta_fh=-|\nabla h|^2+h_t-\mathcal{D}=-\frac{F}{\lambda t}-\left(1-\frac{1}{\lambda}\right)|\nabla h|^2.
\end{equation}
Substituting  \eqref{impform} into \eqref{bake-em2} yields
\begin{equation*}
\begin{aligned}
\left(\Delta_f-\frac{\partial}{\partial t}\right)F&\geq-\frac{F}{t}
+\frac{2t}{m+n}(\Delta_f h)^2-2(m+n-1)Kt|\nabla h|^2\\
&\quad-2\langle\nabla h,\nabla F\rangle+2(\lambda-1)t\langle\nabla h,\nabla\mathcal{D}\rangle
+\lambda t\Delta_f\mathcal{D}.
\end{aligned}
\end{equation*}
Further using \eqref{identity}, the above inequality becomes
\begin{equation}
\begin{aligned}\label{impfor2}
\left(\Delta_f-\frac{\partial}{\partial t}\right)F&\geq-\frac{F}{t}
+\frac{2t}{m+n}(|\nabla h|^2+\mathcal{D}-h_t)^2-2(m+n-1)Kt|\nabla h|^2\\
&\quad-2\langle\nabla h,\nabla F\rangle+2(\lambda-1)t\langle\nabla h,\nabla\mathcal{D}\rangle
+\lambda t\Delta_f\mathcal{D}.
\end{aligned}
\end{equation}
In the following we will compute the last two terms in the inequality \eqref{impfor2}. We
first notice that
\begin{equation*}
\begin{aligned}
2(\lambda-1)t&\left\langle\nabla h,\nabla\big(p(e^h)^{\alpha-1}\big)\right\rangle
+\lambda t\Delta_f\big(p(e^h)^{\alpha-1}\big)\\
&=2(\lambda\alpha-1)t(e^h)^{\alpha-1}\nabla p\nabla h
+(\alpha-1)(\lambda\alpha+\lambda-2)tp(e^h)^{\alpha-1}|\nabla h|^2\\
&\quad+\lambda t(e^h)^{\alpha-1}\Delta_fp
+\lambda(\alpha-1)tp(e^h)^{\alpha-1}\Delta_fh\\
&=2(\lambda\alpha-1)t(e^h)^{\alpha-1}\nabla p\nabla h+\lambda t(e^h)^{\alpha-1}\Delta_fp
-(\alpha-1)p(e^h)^{\alpha-1}F\\
&\quad+(\alpha-1)(\lambda\alpha-1)tp(e^h)^{\alpha-1}|\nabla h|^2,
\end{aligned}
\end{equation*}
where in the last equality we have used the formulae \eqref{identity}. Similar to the above
equality, we also have
\begin{equation*}
\begin{aligned}
2(\lambda-1)t&\left\langle\nabla h,\nabla\big(q(e^h)^{\beta-1}\big)\right\rangle
+\lambda t\Delta_f\big(q(e^h)^{\beta-1}\big)\\
&=2(\lambda\beta-1)t(e^h)^{\beta-1}\nabla q\nabla h+\lambda t(e^h)^{\beta-1}\Delta_fq
-(\beta-1)q(e^h)^{\beta-1}F\\
&\quad+(\beta-1)(\lambda\beta-1)tq(e^h)^{\beta-1}|\nabla h|^2.
\end{aligned}
\end{equation*}
Combining the above two equalities, we have
\begin{equation*}
\begin{aligned}
2(\lambda-1)t&\langle\nabla h,\nabla\mathcal{D}\rangle+\lambda t\Delta_f\mathcal{D}\\
&=2(\lambda\alpha-1)t(e^h)^{\alpha-1}\nabla p\nabla h+\lambda t(e^h)^{\alpha-1}\Delta_fp
-(\alpha-1)p(e^h)^{\alpha-1}F\\
&\quad+(\alpha-1)(\lambda\alpha-1)tp(e^h)^{\alpha-1}|\nabla h|^2\\
&\quad+2(\lambda\beta-1)t(e^h)^{\beta-1}\nabla q\nabla h+\lambda t(e^h)^{\beta-1}\Delta_fq
-(\beta-1)q(e^h)^{\beta-1}F\\
&\quad+(\beta-1)(\lambda\beta-1)tq(e^h)^{\beta-1}|\nabla h|^2\\
&\quad+2(\lambda-1)t\nabla \mu\nabla h+\lambda t\Delta_f\mu.
\end{aligned}
\end{equation*}
Finally substituting this into \eqref{impfor2} gives the proof of the lemma.
\end{proof}

In the following, we will apply Lemma \ref{Lem3.1} and the localized technique of
Li-Yau \cite{[Li-Yau]} and the author \cite{[Wu10]} to give parabolic gradient estimates
for the positive smooth solutions to the equation \eqref{Yameq} on smooth metric measure
spaces.
\begin{proof}[Proof of Theorem \ref{mainpara}]
Firstly, we introduce an auxiliary cut-off function and its useful properties. This
cut-off function is very important in the following proof.

We choose any $C^2$ cut-off function $\tilde{\varphi}$ on $[0,\infty)$ such that
$\tilde{\varphi}(r)\equiv1$ for $r\in[0,1]$, $\tilde{\varphi}(r)=0$ for $r\in[2,\infty)$,
and $0\leq\tilde{\varphi}(r)\leq1$; meanwhile $\tilde{\varphi}$ satisfies
\[
-c_1\leq\frac{\tilde{\varphi}'(r)}{\tilde{\varphi}^{1/2}(r)}\leq 0
\quad\mathrm{and}\quad
\tilde{\varphi}''(r)\geq-c_1
\]
for some universal positive constant $c_1$. Let
\[
\varphi(x)=\tilde{\varphi}\left(\frac{r(x)}{R}\right),
\]
where $r(x)$ denotes the distance between $x$ and $x_0$ in $M$. Then
$\mathrm{supp}\varphi\subseteq B(x_0,2R)$ and $\varphi|_{B(x_0,R)}\equiv1$.
We shall consider the function $\varphi F$ in $H_{2R,T}$. By the argument of
Calabi \cite{[Cala]}, by using approximation, we can assume without loss of
generality that that $\varphi(x)\in C^2(M)$ with support in $B(x_0,2R)$.
By a easy computation, we have
\begin{equation}\label{formu0}
\frac{|\nabla\varphi|^2}{\varphi}\leq\frac{c_1^2}{R^2}
\end{equation}
and
\begin{equation}\label{formu01}
\Delta_f\varphi=\frac{\tilde{\varphi}'
\Delta_fr}{R}+\frac{\tilde{\varphi}''|\nabla r|^2}{R^2}
\end{equation}
in $B(x_0,2R)$.
On the other hand, since $\mathrm{Ric}^m_f\geq-(m+n-1)K$ for some
$K\geq 0$, the generalized Laplacian comparison theorem (see
\cite{[WW]}) gives that
\[
\Delta_fr\leq(m+n-1)\sqrt{K}\coth(\sqrt{K}\,r).
\]
Since $\coth$ is decreasing, and ${\tilde{\varphi}}'=0$ when $r(x)<R$, by
\eqref{formu01}, this implies
\begin{equation}
\begin{aligned}\label{formu3}
\Delta_f\varphi&\geq-\frac{c_1}{R}(m+n-1)\sqrt{K}\coth(\sqrt{K}R)-\frac{c_1}{R^2}\\
&\geq-\frac{(m+n)c_1(1+R\sqrt{K})}{R^2},
\end{aligned}
\end{equation}
where we have used the inequality $\sqrt{K}\coth(\sqrt{K}R)\leq\frac 1R(1+\sqrt{K}R)$.

Secondly, we will apply the $f$-Laplacian operator $\Delta_f$ to the function
$\varphi F$ and get a useful inequality. Then we apply the maximum principle argument
to the inequality in a compactly supported set and obtain the Li-Yau gradient estimate.

For any $0<\tau\leq T$, if $\varphi F\leq0$ in $H_{2R,\tau}$, then the desired
estimate follows. Now we assume $\max_{(x,t)\in H_{2R,\tau}}(\varphi F)>0$.
Let $(x_1,t_1)$ be a point where $\varphi F$ achieves the positive maximum, where
$x_1\in B(x_0,2R)$ and $0<t_1\leq\tau$. Clearly, at $(x_1,t_1)$, we have
\begin{equation}\label{formu4}
\nabla(\varphi F)=0,\quad F_t\geq0\quad\mathrm{and}\quad \Delta_f(\varphi
F)\leq0.
\end{equation}
From now on all calculations below will be at $(x_1,t_1)$. Applying Lemma \ref{Lem3.1} to the
following equality
\[
\Delta_f(\varphi F)=F(\Delta_f\varphi)+2\langle\nabla\varphi,\nabla
F\rangle+\varphi(\Delta_fF),
\]
and using \eqref{formu0}, \eqref{formu01}, \eqref{formu3}, \eqref{formu4} and
the fact $e^h=u$, we get that
\begin{equation*}
\begin{aligned}
0&\geq\Delta_f(\varphi F)\\
&\geq-F\frac{(m+n)c_1(1+R\sqrt{K})}{R^2}-2F\frac{|\nabla\varphi|^2}{\varphi}\\
&\quad+\varphi\left[-\frac{F}{t_1}-2\langle\nabla h,\nabla F\rangle
-(\alpha-1)pu^{\alpha-1}F-(\beta-1)qu^{\beta-1}F\right]\\
&\quad+\frac{2t_1\varphi}{m+n}\Big[|\nabla h|^2+pu^{\alpha-1}+qu^{\beta-1}+\mu-h_t\Big]^2\\
&\quad+t_1\varphi|\nabla h|^2\Big[(\alpha-1)(\lambda\alpha-1)pu^{\alpha-1}
+(\beta-1)(\lambda\beta-1)qu^{\beta-1}-2(m+n-1)K\Big]\\
&\quad+2t_1\varphi\Big[(\lambda\alpha-1)u^{\alpha-1}\langle\nabla h,\nabla p\rangle
+(\lambda\beta-1)u^{\beta-1}\langle\nabla h,\nabla q\rangle
+(\lambda-1)\langle\nabla h,\nabla\mu\rangle\Big]\\
&\quad+\lambda t_1\varphi\Big[u^{\alpha-1}\Delta_fp+u^{\beta-1}\Delta_fq+\Delta_f\mu\Big]\\
&\geq F\left[-\frac{(m+n)c_1(1+R\sqrt{K})+2c^2_1}{R^2}-\frac{\varphi}{t_1}
-(\alpha-1)pu^{\alpha-1}\varphi-(\beta-1)qu^{\beta-1}\varphi\right]\\
&\quad+2F\langle\nabla h,\nabla\varphi\rangle+\frac{2t_1\varphi}{m+n}\Big[|\nabla h|^2+pu^{\alpha-1}+qu^{\beta-1}+\mu-h_t\Big]^2\\
&\quad+t_1\varphi|\nabla h|^2\Big[(\alpha-1)(\lambda\alpha-1)pu^{\alpha-1}
+(\beta-1)(\lambda\beta-1)qu^{\beta-1}-2(m+n-1)K\Big]\\
&\quad+2t_1\varphi\Big[(\lambda\alpha-1)u^{\alpha-1}\langle\nabla h,\nabla p\rangle
+(\lambda\beta-1)u^{\beta-1}\langle\nabla h,\nabla q\rangle
+(\lambda-1)\langle\nabla h,\nabla\mu\rangle\Big]\\
&\quad+\lambda t_1\varphi\Big[u^{\alpha-1}\Delta_fp+u^{\beta-1}\Delta_fq+\Delta_f\mu\Big].
\end{aligned}
\end{equation*}

Multiplying both sides of the above inequality by $t_1\varphi$, using the assumptions of
$p(x,t)$, $q(x,t)$, $\mu(x,t)$ and $\varphi(x)$ in Theorem \ref{main}, recalling that
$0\leq\varphi\leq1$, we in fact get that
\begin{equation}
\begin{aligned}\label{formu6}
0&\geq-t_1\varphi F\left[\frac{(m{+}n)c_1(1+R\sqrt{K})+2c^2_1}{R^2}
+\frac{1}{t_1}+[(\alpha-1)p]^+u^{\alpha-1}+[(\beta-1)q]^+u^{\beta-1}\right]\\
&\quad-2c_1R^{-1}t_1F|\nabla h|\varphi^{3/2}+\frac{2t^2_1\varphi^2}{m+n}\Bigg\{\Big[|\nabla h|^2+pu^{\alpha-1}+qu^{\beta-1}+\mu-h_t\Big]^2\\
&\quad+\frac{m{+}n}{2}|\nabla h|^2\Big[[(\alpha{-}1)(\lambda\alpha{-}1)p]^-u^{\alpha-1}
{+}[(\beta{-}1)(\lambda\beta{-}1)q]^-u^{\beta-1}{-}2(m{+}n{-}1)K\Big]\Bigg\}\\
&\quad-2t_1^2\varphi^{1/2}|\nabla h|\Big[|\lambda\alpha-1|a_1u^{\alpha-1}+|\lambda\beta-1|a_2
u^{\beta-1}+(\lambda-1)a_3\Big]\\
&\quad+\lambda t_1^2\left[\inf_{H_{2R,T}}\left(u^{\alpha-1}b_1+u^{\beta-1}b_2\right)+b_3\right].
\end{aligned}
\end{equation}
In the above inequality, we denote
\[
p^+:=\sup_{(x,t)\in H_{R,T}}\{p^+(x,t),0\}\quad \mathrm{and}\quad
p^-:=\inf_{(x,t)\in H_{R,T}}\{p^-(x,t),0\},
\]
for any $p(x,t)\in C^\infty(H_{R,T})$, where
\[
p^+(x,t):=\max\{p(x,t),0\}\quad \mathrm{and}\quad p^-(x,t):=\min\{p(x,t),0\}.
\]
We let
\[
y:=\varphi|\nabla h|^2
\]
and
\[
z:=\varphi(h_t-pu^{\alpha-1}-qu^{\beta-1}-\mu).
\]
Then \eqref{formu6} can be rewritten as
\begin{equation}
\begin{aligned}\label{formu7}
0&\geq-\varphi F\left[\frac{(m+n)c_1(1+R\sqrt{K})+2c^2_1}{R^2}t_1+1\right]\\
&\quad-\varphi F\left[[(\alpha-1)p]^+\sup_{H_{2R,T}}\{u^{\alpha-1}\}t_1
+[(\beta-1)q]^+\sup_{H_{2R,T}}\{u^{\beta-1}\}t_1\right]\\
&\quad+\frac{2t^2_1}{m+n}\Bigg\{(y-z)^2-c_1(m+n)R^{-1}y^{1/2}(y-\lambda z)-(m+n)\tilde{K}y-(m+n)\gamma y^{1/2}\Bigg\}\\
&\quad+\lambda t_1^2\left[\inf_{H_{2R,T}}\left(u^{\alpha-1}b_1+u^{\beta-1}b_2\right)+b_3\right],
\end{aligned}
\end{equation}
where
\[
\tilde{K}{:=}(n+m-1)K-\frac 12[(\alpha-1)(\lambda\alpha-1)p]^-\kern-3pt\sup_{H_{2R,T}}\{u^{\alpha-1}\}
-\frac 12[(\beta-1)(\lambda\beta-1)q]^-\kern-3pt\sup_{H_{2R,T}}\{u^{\beta-1}\}
\]
and
\[
\gamma:=|\lambda\alpha-1|a_1\sup_{H_{2R,T}}\{u^{\alpha-1}\}+|\lambda\beta-1|a_2\sup_{H_{2R,T}}\{u^{\beta-1}\}+(\lambda-1)a_3.
\]
Inequality \eqref{formu7} is rather complicated and we want to simplify it so that
the inequality can be estimated efficiently. Indeed, we can follow the Li-Yau's
arguments \cite{[Li-Yau]} to estimate the third line of inequality \eqref{formu7}.
The similar argument also appeared in \cite{[Wu10]}. That is to say, we can use
the Cauchy-Schwarz inequality to get the following key inequality
\begin{equation*}
\begin{aligned}
(y-z)^2-&c_1(m+n)R^{-1}y^{1/2}(y-\lambda z)-(m+n)\tilde{K}y
-(m+n)\gamma y^{1/2}\\
&\geq \lambda^{-2}(y-\lambda
z)^2-\frac{(m+n)^2}{8}c^2_1\lambda^2(\lambda-1)^{-1}R^{-2}(y-\lambda z)\\
&\quad-\frac34 4^{-\frac13}(m+n)^{4/3}
\gamma^{4/3}\left(\frac{\lambda}{\lambda-1}\right)^{2/3}\varepsilon^{-1/3}
-\frac{(m+n)^2\lambda^2\tilde{K}^2}{4(1-\varepsilon)(\lambda-1)^2}
\end{aligned}
\end{equation*}
for any $0<\varepsilon<1$. Substituting this inequality into \eqref{formu7}
and arranging the terms yields
\begin{equation*}
\begin{aligned}
0&\geq-\varphi F\left[\frac{(m+n)c_1(1+R\sqrt{K})+2c^2_1}{R^2}t_1+1\right]\\
&\quad-\varphi F\left[[(\alpha-1)p]^+\sup_{H_{2R,T}}\{u^{\alpha-1}\}t_1
+[(\beta-1)q]^+\sup_{H_{2R,T}}\{u^{\beta-1}\}t_1\right]\\
&\quad+\frac{2}{m+n}\left[\lambda^{-2}(\varphi F)^2-\frac{(m+n)^2c^2_1\lambda^2}{8(\lambda-1)R^2}t_1(\varphi F)\right]\\
&\quad+\frac{t_1^2}{m+n}\left[\frac{-3}{2}
\left(\frac{(m+n)^4\lambda^2}{4\varepsilon(\lambda-1)^2}\right)^{\frac13}\gamma^{\frac43}
-\frac{(m+n)^2\lambda^2\tilde{K}^2}{2(1-\varepsilon)(\lambda-1)^2}\right]\\
&\quad+\lambda t_1^2\left[\inf_{H_{2R,T}}\left(u^{\alpha-1}b_1+u^{\beta-1}b_2\right)+b_3\right]\\
&=\frac{2\lambda^{-2}}{m+n}(\varphi F)^2-\Phi\cdot(\varphi F)-t_1^2\Psi,
\end{aligned}
\end{equation*}
where
\begin{equation*}
\begin{aligned}
\Phi:&=\frac{(m+n)c_1(1+R\sqrt{K})+2c^2_1}{R^2}t_1+\frac{(m+n)c^2_1\lambda^2}{4(\lambda-1)R^2}t_1+1\\
&\quad+[(\alpha-1)p]^+\sup_{H_{2R,T}}\{u^{\alpha-1}\}t_1+[(\beta-1)q]^+\sup_{H_{2R,T}}\{u^{\beta-1}\}t_1
\end{aligned}
\end{equation*}
and
\[
\Psi:=\frac32\left(\frac{(m+n)\lambda^2}{4\varepsilon(\lambda-1)^2}\right)^{\frac13}\gamma^{\frac43}
+\frac{(m+n)\lambda^2\tilde{K}^2}{2(1-\varepsilon)(\lambda-1)^2}
-\lambda\left[\inf_{H_{2R,T}}\left(u^{\alpha-1}b_1+u^{\beta-1}b_2\right)+b_3\right],
\]
and where
\[
\gamma:=|\lambda\alpha-1|a_1\sup_{H_{2R,T}}\{u^{\alpha-1}\}+|\lambda\beta-1|a_2\sup_{H_{2R,T}}\{u^{\beta-1}\}+(\lambda-1)a_3
\]
and
\[
\tilde{K}{:=}(n+m-1)K-\frac 12[(\alpha-1)(\lambda\alpha-1)p]^-\kern-3pt\sup_{H_{2R,T}}\{u^{\alpha-1}\}
-\frac 12[(\beta-1)(\lambda\beta-1)q]^-\kern-3pt\sup_{H_{2R,T}}\{u^{\beta-1}\}.
\]
This implies
\begin{equation}
\begin{aligned}\label{fproof}
(\varphi F)(x_1,t_1)&\leq \frac{m+n}{4}\lambda^2\left[\Phi+
\left(\Phi^2+\frac
{8}{m+n}\lambda^{-2}t_1^2\Psi\right)^{1/2}\right]\\
&\leq\frac{m+n}{4}\lambda^2\left[\Phi+\Phi+\left(\frac
{8}{m+n}\lambda^{-2}t_1^2\Psi\right)^{1/2}\right]\\
&=\frac{m+n}{2}\lambda^2\Phi+t_1\lambda\Big(\frac{m+n}{2}\Psi\Big)^{1/2},
\end{aligned}
\end{equation}
where $\Phi$ and $\Psi$ are defined as above. Notice that on $B(x_0,R)\times[0,\tau]$,
since $\varphi\equiv1$ and $(x_1,t_1)$ is a maximum point of function $\varphi F$,
we have
\begin{equation}\label{fpro}
\sup\limits_{B(x_0,R)}F(x,\tau)\leq(\varphi F)(x_1, t_1).
\end{equation}
Substituting \eqref{fproof} into \eqref{fpro}, and using a easy fact that
$t_1\leq\tau$, we indeed show that
\begin{equation*}
\begin{aligned}
&\tau\cdot\sup\limits_{B(x_0,R)}\Big[|\nabla h|^2+\lambda pu^{\alpha-1}
+\lambda qu^{\beta-1}+\lambda\mu-\lambda h_t\Big](x,\tau)\\
&\leq\tau\lambda\Big(\frac{m+n}{2}\Psi\Big)^{1/2}
+\frac{m+n}{2}\lambda^2\left[\frac{(m+n)c_1(1+R\sqrt{K})+2c^2_1}{R^2}\tau
+\frac{(m+n)c^2_1\lambda^2}{4(\lambda-1)R^2}\tau\right]\\
&\quad+\frac{m+n}{2}\lambda^2+\frac{m+n}{2}\lambda^2
\left\{\left[(\alpha-1)p\right]^+\sup_{H_{2R,T}}\{u^{\alpha-1}\}\tau
+\left[(\beta-1)q\right]^+\sup_{H_{2R,T}}\{u^{\beta-1}\}\tau\right\},
\end{aligned}
\end{equation*}
which immediately implies the theorem because $\tau\in(0,T]$ is arbitrary.
\end{proof}

\section{Parabolic Liouville theorem}\label{sec4}
In this section, we will apply Theorem \ref{main} to give many sufficient conditions
on the growth of solutions and coefficients that guarantee the parabolic Liouville
theorems for various cases of the equation \eqref{Yameq}.

\medskip

First, we will prove Theorem \ref{app1} in the introduction. We consider the
case: $\alpha>1$, $\mu(x,t)\equiv \mu(x)$, $p(x,t)\equiv p(x)\not\equiv0$
and $q(x,t)\equiv 0$ in the equation \eqref{Yameq}.
\begin{proof}[Proof of Theorem \ref{app1}]
Under the assumptions of Theorem \ref{app1}, let $u(x,t)$ be a positive smooth
ancient solution to the equation \eqref{Yameqapp1}. For any fixed space-time point $(x_0,t_0)$,
since $\alpha>1$ and $K=0$, we apply Theorem \ref{main} to $u(x_0,t_0)$ in the
space-time set $B(x_0, R)\times(t_0-R, t_0]$ (i.e., let $T=R$ in $Q_{R,T}$), and obtain that
\begin{equation*}
\begin{aligned}
|\nabla\ln u(x_0,t_0)|&\le c(n)\left(1+\ln(D(Q_{R,R}))-\ln u(x_0,t_0)\right)\\
&\quad\times\left[\frac{1+\sqrt{\sigma^+}}{\sqrt{R}}+o(R^{-\frac{s}{2}})+o(R^{-\frac{s}{3}})+o[R^{(\widetilde{\kappa}-\kappa)\frac{\alpha-1}{2}}]
+\,o[R^{(\widetilde{\kappa}-\kappa)\frac{\alpha-1}{3}}]\right]
\end{aligned}
\end{equation*}
for sufficiently large $R>>2$, depending on $|t_0|$, where $R$ has been chosen sufficiently large
such that $R\geq |t_0|$. Since $u(x,t)=o(\left[r(x)+|t|\,\right]^{\widetilde{\kappa}})$ in $Q_{R,R}$,
then we have $D(Q_{R,R})=o(R^{\widetilde{\kappa}})$. For the number $\widetilde{\kappa}\in(0,\kappa)$ and
the fixed value $\ln u(x_0,t_0)$, letting $R\to \infty$ in the above inequality, we immediately get
\[
|\nabla u(x_0, t_0)|=0.
\]
Since $(x_0, t_0)$ was chosen arbitrarily, we conclude that $u(x,t)=u(t)$ for all $x\in M$.

\textbf{Case One}: $\mu(x)\equiv0$.

In this case \eqref{Yameqapp1} becomes
\[
u'(t)=p(x)u^\alpha(t),\quad p(x)\not\equiv0,\,\, \alpha>1.
\]
This equation implies $p(x)\equiv c$ for some constant $c<0$ due to the growth assumption on $p(x)$.
Therefore,
\[
u^{1-\alpha}(t)=c(1-\alpha)t+u^{1-\alpha}(0).
\]
Since $u$ is a positive ancient solution, from above we see that $u^{1-\alpha}(-\infty)<0$ for $t\to-\infty$.
This is a contradiction with the positivity of $u(x,t)$.

\textbf{Case Two}: $\mu(x)\not\equiv0$.

In this case, \eqref{Yameqapp1} reduces to
\[
u'(t)=\mu(x)u(t)+p(x)u^\alpha(t),\quad \mu(x)\not\equiv0,\,\,p(x)\not\equiv0,\,\,\alpha>1,
\]
which can be rewritten as a first-order ODE by
\[
[u^{1-\alpha}(t)]'=(1-\alpha)p(x)+(1-\alpha)\mu(x)u^{1-\alpha}.
\]
This equation has a general solution
\begin{equation}\label{ODE}
u^{1-\alpha}(t)=Ce^{(1-\alpha)\mu(x)t}-p(x)/\mu(x),
\end{equation}
where $C$ is a arbitrary constant, $\mu(x)\not\equiv0$ and $p(x)\not\equiv0$.

Since the left-hand side of \eqref{ODE} is independent of $x$, it must hold that
$p(x)/\mu(x)$ is constant. Moreover, if $\mu(x)\equiv c<0$ ($c>0$ is impossible due to the growth of $\mu(x)$),
then $p(x)\equiv c'<0$ ($c'>0$ is impossible due to the growth of $p(x)$).
In this case, \eqref{ODE} becomes
\[
u^{1-\alpha}(t)=\left(u^{1-\alpha}(0)+\frac{c}{c'}\right)e^{(1-\alpha)c\,t}-\frac{c}{c'},\quad t<0,
\]
where $\alpha>1$, $c<0$, $u(0)>0$ and $c/c'>0$. Letting $t\to-\infty$, we get
\[
u^{1-\alpha}(t)\to-\frac{c}{c'}<0,
\]
which is impossible since $u(x,t)>0$. So $\mu(x)$ is not constant and from \eqref{ODE},
we conclude that $C\equiv0$ and $u^{1-\alpha}(t)=-p(x)/\mu(x)$ is constant. Therefore
$\mu(x)\equiv-cp(x)$ for some constant $c>0$ and $u(x,t)\equiv c^{\frac{1}{\alpha-1}}$.
\end{proof}

\medskip

In Theorem \ref{app1}, if $\mu(x)$ and $p(x)$ are both negative constants, then they
naturally satisfy the conditions (1) and (2). In this case we are able to improve the
growth condition of $u(x,t)$ and get a simple statement, which was also proved by Dung,
Khanh and Ngo (see Corollary 2.6 in \cite{[DKN]}).

\begin{coro}\label{appthm3}
Let $(M,g,e^{-f}dv)$ be an $n$-dimensional complete smooth metric measure space with
$\mathrm{Ric}_f\geq0$.  There does not exist any positive ancient solution to equation
\begin{equation}\label{Yameqcor3}
\left(\Delta_f-\frac{\partial}{\partial t}\right)u+\mu\,u+p\,u^\alpha=0,\quad \alpha>1,\,\,\mu<0,\,\,p<0,
\end{equation}
such that $u(x,t)=e^{o(r^{\frac 12}(x)+|t|^{\frac 12})}$ near infinity. Moreover,
if $f$ is identically constant, then the growth of $u$ can be relaxed to
$u(x,t)=e^{o(r(x)+|t|^{\frac 12})}$.
\end{coro}
\begin{proof}[Proof of Corollary \ref{appthm3}]
Because $\mu(x)$ and $p(x)$ are both negative constants, we know that the conditions
(1) and (2) in Theorem \ref{app1} naturally hold. Now let $u(x,t)$ be a positive smooth
ancient solution to the equation \eqref{Yameqcor3}, such that
\[
\ln u(x,t)=o(r^{\frac 12}(x)+|t|^{\frac 12})
\]
near infinity. Similar to the proof of Theorem \ref{app1}, for a fixed space-time point $(x_0,t_0)$, we apply
Theorem \ref{main} (i) to $u(x_0,t_0)$ in $Q_{R,R}=B(x_0, R)\times(t_0-R, t_0]$,
\begin{equation}\label{cgradest1}
|\nabla\ln u(x_0,t_0)|\le c(n)\left(1+o(\sqrt{R})-\ln u(x_0,t_0)\right)
\left[\frac{1}{\sqrt{R}}+\sqrt{\frac{\sigma^+}{R}}\,\right]
\end{equation}
for sufficiently large $R>>2$, depending on $|t_0|$. Then letting $R\to \infty$, we have
$|\nabla u(x_0, t_0)|=0$. Since $(x_0,t_0)$ is arbitrary, we get $u(x,t)=u(t)$ for all
$x\in M$. Finally, the conclusion follows by the same argument of Theorem \ref{app1}.

As for the case $f$ is constant, we assume that
\[
\ln u(x,t)=o(r(x)+|t|^{\frac 12})
\]
near infinity. We apply Theorem \ref{main} to $u(x_0,t_0)$ in
$Q_{R,R^2}=B(x_0, R)\times(t_0-R^2, t_0]$ and the proof is almost
the same as before except the corresponding gradient
estimate of \eqref{cgradest1} is replaced by
\[
|\nabla\ln u(x_0,t_0)|\le c(n)\left(1+o(R)-\ln u(x_0,t_0)\right)\cdot\frac 1R
\]
for sufficiently large $R$, depending on $|t_0|$.
\end{proof}

\medskip

Second, we consider the case: $\alpha=1$, $\mu(x,t)\equiv\mu(x)$, $p(x,t)\equiv p(x)$
and $q(x,t)\equiv0$ in the equation \eqref{Yameq}. In this case we prove that
\begin{theorem}\label{app2}
Let $(M,g,e^{-f}dv)$ be an $n$-dimensional complete smooth metric measure space with
$\mathrm{Ric}_f\geq0$. Assume that $\mu(x)$ in the following equation
\begin{equation}\label{Yameqapp2}
\left(\Delta_f-\frac{\partial}{\partial t}\right)u+\mu(x)u=0
\end{equation}
satisfies
\[
\mu^+|_{B(x_0,R)}=o(R^{-1})\quad \mathrm{and} \quad
\sup_{B(x_0,R)}|\nabla\mu|=o(R^{-\frac 32}),
\,\,\mathrm{as} \,\, R\to\infty.
\]
\begin{enumerate}
\item For $\mu(x)\not\equiv0$, there does not exist any positive ancient solution to the equation \eqref{Yameqapp2}
such that $u(x,t)=e^{o(r^{\frac 12}(x)+|t|^{\frac 12})}$ near infinity;
\item for $\mu(x)\equiv0$, there only exist constant positive ancient solution to the equation \eqref{Yameqapp2}
such that $u(x,t)=e^{o(r^{\frac 12}(x)+|t|^{\frac 12})}$ near infinity.
\end{enumerate}
\end{theorem}
\begin{remark}
There indeed exist many functions $\mu(x)$ satisfying the growth of $\mu$, such
as $\mu(x)=-e^{-x}/(x^2+1)$ in $\mathbb{R}^1$. If $\mu(x)$ is negative constant,
it naturally satisfies the growth of $\mu$. If $\mu(x)\equiv0$, the theorem
returns to a slight improvement of \cite{[Wu15]}. Notice that the growth condition
of $u$ is necessary. For example, let $u=e^{x+t}$, $f=-x$ and $\mu(x)=-1$ in
$\mathbb{R}^1$. Then $u$ is a positive eternal solution to the equation \eqref{Yameqapp2}.
\end{remark}

\begin{proof}[Proof of Theorem \ref{app2}]
Let $u(x,t)$ be a positive smooth ancient solution to the equation \eqref{Yameqapp2},
such that
\[
\ln u(x,t)=o(r^{\frac 12}(x)+|t|^{\frac 12})
\]
near infinity. For any point $(x_0,t_0)$, since $\alpha=1$ and $K=0$, applying Theorem \ref{main}
to $u(x_0,t_0)$ in the set $Q_{R,R}:=B(x_0, R)\times(t_0-R, t_0]$,
\[
|\nabla\ln u(x_0,t_0)|\le c(n)\left(1+o(\sqrt{R})-\ln u(x_0,t_0)\right)\left[\frac{1+\sqrt{\sigma^+}}{\sqrt{R}}
+o(R^{-\frac 12})\right]
\]
for sufficiently large $R>>2$, depending on $|t_0|$. Letting $R\to \infty$,
$|\nabla u(x_0, t_0)|=0$. Since $(x_0, t_0)$ is arbitrary, we know that
$u(x,t)=u(t)$ for all $x\in M$, which satisfies
\[
u'(t)=\mu(x)u(t).
\]
If $\mu(x)\equiv0$, then $u(x,t)\equiv c$ is positive constant. If $\mu(x)\not\equiv0$,
this implies $\mu(x)=C$ for some constant $C<0$ by the growth of $\mu(x)$.
So we have
\[
u(t)=u(0)e^{Ct},\quad t<0.
\]
This contradicts the assumption of theorem
$u(x,t)=e^{o(r^{\frac 12}(x)+|t|^{\frac 12})}$ near infinity. Hence the theorem follows.
\end{proof}

\medskip

In Theorem \ref{app2}, if we further assume $f$ is a constant, we
can improve the growth assumptions on $\mu(x)$ and $u(x,t)$.
This has also been obtained by Zhu \cite{[Zhu2]}.
\begin{coro}\label{corapp2}
Let $(M,g)$ be an $n$-dimensional complete noncompact Riemannian manifold with
$\mathrm{Ric}\geq0$. Assume that $\mu(x)$ in the following equation
\begin{equation}\label{cormeqapp2}
\left(\Delta-\frac{\partial}{\partial t}\right)u+\mu(x)u=0
\end{equation}
satisfies
\[
\mu^+|_{B(x_0,R)}=o(R^{-2})\quad \mathrm{and} \quad
\sup_{B(x_0,R)}|\nabla\mu|=o(R^{-3}),
\,\,\mathrm{as} \,\, R\to\infty.
\]
\begin{enumerate}
\item For $\mu(x)\not\equiv0$, there does not exist any positive ancient solution to the equation \eqref{cormeqapp2}
such that $u(x,t)=e^{o(r(x)+|t|^{\frac 12})}$ near infinity;
\item for $\mu(x)\equiv0$, there only exist constant positive ancient solution to the equation \eqref{cormeqapp2}
such that $u(x,t)=e^{o(r(x)+|t|^{\frac 12})}$ near infinity.
\end{enumerate}

\end{coro}
\begin{proof}[Proof of Corollary \ref{corapp2}]
The proof is nearly the same as the proof Theorem \ref{app2} with the only difference
is that we apply Theorem \ref{main} to $u(x_0,t_0)$ in the new space-time set
$Q_{R,R^2}=B(x_0, R)\times(t_0-R^2, t_0]$, and get that
\[
|\nabla\ln u(x_0,t_0)|\le c(n)\left(1+o(R)-\ln u(x_0,t_0)\right)\left[\frac{1}{R}
+o(R^{-1})\right]
\]
for sufficiently large $R>>2$, depending on $|t_0|$. We would like to point out that
the term $\sqrt{\frac{\sigma^+}{R}}$ in Theorem \ref{main} does not exist in
this case (see Remark \ref{rem1}).
\end{proof}
\begin{remark}
The growth condition of $u(x,t)$ is sharp in the space direction. For example, let
$u=e^{2x+t}$ and $\mu(x)=-3$ in $\mathbb{R}^1$. Obviously, $u$ is a positive
eternal solution to the equation \eqref{cormeqapp2}.
\end{remark}

\medskip

Third, we prove Theorem \ref{app3} in the introduction. Let $\alpha<1$,
$\mu(x,t)\equiv\mu(x)\not\equiv0$, $p(x,t)\equiv p(x)\not\equiv0$ and
$q(x,t)\equiv q(x)\equiv0$ in \eqref{Yameq}, and we have
\begin{proof}[Proof of Theorem \ref{app3}]
Let $u(x,t)$ be a positive ancient solution to the equation \eqref{Yameqapp3} such that
\[
(r(x)+|t|)^{-\widetilde{\kappa}}\leq u(x,t)\leq (r(x)+|t|)^{\delta}
\]
for some $\widetilde{\kappa}\in(0,\kappa)$ and $\delta>0$ near infinity. For any point $(x_0,t_0)$,
since $\alpha<1$ and $K=0$, applying Theorem \ref{main} to $u(x_0,t_0)$ in
$B(x_0, R)\times(t_0-R, t_0]$, we get that
\begin{equation}
\begin{aligned}\label{compinq}
&|\nabla\ln u(x_0,t_0)|\le c(n)\Big(1+c(\delta)\,\ln R-\ln u(x_0,t_0)\Big)\\
&\quad\times\left[\frac{1{+}\sqrt{\sigma^+}}{\sqrt{R}}+o(R^{-\frac{s}{2}})+o(R^{-\frac{s}{3}})
+o[R^{-\frac {\kappa}{2}(1-\alpha)}]\, (R^{-\widetilde{\kappa}})^{\frac{\alpha-1}{2}}
+\,o[R^{-\frac{\kappa}{3}(1-\alpha)}]\, (R^{-\widetilde{\kappa}})^{\frac{\alpha-1}{3}}\right]
\end{aligned}
\end{equation}
for sufficiently large $R>>2$, depending on $|t_0|$, where $R$ has been chosen sufficiently
large such that $R\geq |t_0|$. We would like to point out, in the above complicated
estimate, we have chosen $\min_{\mathcal{Q}_{R,T}}u(x,t)=(3R)^{-\widetilde{\kappa}}$ due to the
fact: $r(x)\leq R$ and
$|t|\leq |t_0|+R\leq 2R$.

Letting $R\to \infty$ in \eqref{compinq}, since $\widetilde{\kappa}\in(0,\kappa)$, we have
\[
|\nabla u(x_0, t_0)|=0.
\]
Since $(x_0, t_0)$ is arbitrary, we conclude that $u(x,t)=u(t)$ and it satisfies
\begin{equation}\label{Yamappc1}
u'(t)=\mu(x)u(t)+p(x)u^\alpha,\quad p(x)\not\equiv0,\,\, \alpha<1.
\end{equation}

\textbf{Case One}: $\mu(x)\equiv0$.

In this case \eqref{Yamappc1} becomes
\[
u'(t)=p(x)u^\alpha(t),\quad p(x)\not\equiv0,\,\, \alpha<1.
\]
Similar to the Case One in the proof of Theorem \ref{app1}, this is impossible.

\textbf{Case Two}: $\mu(x)\not\equiv0$.

Equation \eqref{Yamappc1} can be rewritten as a first-order ODE by
\[
[u^{1-\alpha}(t)]'=(1-\alpha)\mu(x)u^{1-\alpha}+(1-\alpha)p(x),
\]
which has a general solution
\[
u^{1-\alpha}(t)=Ce^{(1-\alpha)\mu(x)t}-p(x)/\mu(x),\quad \mu(x)\not\equiv0,\quad p(x)\not\equiv0,
\]
where $C$ is a arbitrary constant.
Similar to the proof of  Case Two in Theorem \ref{app1}, we have
$\mu(x)\equiv-cp(x)$ for some constant $c>0$ and $u^{1-\alpha}(t)=1/c$.
\end{proof}

%%%%%%%%%%%%%%%%%%%%%%%%%%%%%%%%%%%%%%%%%%%%%%%%%%%%%%%%%%%%%%%%%%5

\section{Elliptic Liouville theorem}\label{sec5}

In this section, we have two goals. One is that we apply Theorem \ref{main} to discuss
Liouville-type theorems for some elliptic versions of the equation \eqref{Yameq}
on complete (not necessarily compact) manifolds and smooth metric measure spaces.

\medskip

Firstly, we consider a special elliptic version of \eqref{Yameq} for $\alpha<1$
on a smooth metric measure space, which supplements Yang's result
\cite{[Yang]}.
\begin{theorem}\label{ellap1}
Let $(M,g,e^{-f}dv)$ be an $n$-dimensional complete smooth metric measure space with
$\mathrm{Ric}_f\geq0$. Assume that there exists a constant $\kappa>0$
such that $p(x)$ in the following elliptic equation
\begin{equation}\label{ellqap1}
\Delta_f u+p(x)u^\alpha=0, \quad \alpha<1,\,\,p(x)\not\equiv0,
\end{equation}
satisfies
\[
\sup_{B(x_0,R)}|\,p\,|=o[R^{-\kappa(1-\alpha)}]\quad \mathrm{and}
\quad \sup_{B(x_0,R)}|\nabla p|=o[R^{-\kappa(1-\alpha)}],
\, as\, R\to\infty.
\]
Then there does not exist any positive solution to \eqref{ellqap1} on $M$, such that
\[
r^{-\widetilde{\kappa}}(x)\leq u(x)\leq r^\delta(x)
\]
for some $\widetilde{\kappa}\in(0,\kappa)$ and $\delta>0$ near infinity.
\end{theorem}

\begin{proof}[Proof of Theorem \ref{ellap1}]
The proof is similar to the argument of Theorem \ref{app3}.
Let $u(x,t)$ be a positive smooth solution to the equation \eqref{ellqap1} such that
\[
r^{-\widetilde{\kappa}}(x)\leq u(x)\leq r^\delta(x)
\]
for some $\widetilde{\kappa}\in(0,\kappa)$ and $\delta>0$ near infinity. For any fixed point $x_0$,
since $\alpha<1$ and $K=0$, we apply Theorem \ref{main} to $u(x_0)$ in $B(x_0,R)$
(here function $u$ is independent of time $t$), and get that
\begin{equation*}
\begin{aligned}
|\nabla\ln u(x_0)|&\le c(n)\left(1+c(\delta)\ln R-\ln u(x_0)\right)\\
&\quad\times\left[\frac{1+\sqrt{\sigma^+}}{\sqrt{R}}+o[R^{-\frac{\kappa}{2}(1-\alpha)}]\,
(R^{-\widetilde{\kappa}})^{\frac{\alpha-1}{2}}
+\,o[R^{-\frac{\kappa}{3}(1-\alpha)}]\, (R^{-\widetilde{\kappa}})^{\frac{\alpha-1}{3}}\right]
\end{aligned}
\end{equation*}
for $R>2$, where we have used the fact that $\min_{x\in B(x_0,R)}u(x)=R^{-\widetilde{\kappa}}$.
Letting $R\to \infty$ and using $\kappa>\widetilde{\kappa}>0$, we get
\[
|\nabla u(x_0)|=0.
\]
Since point $x_0$ was chosen arbitrarily, we have that $u(x)\equiv c$ for some constant
$c>0$. Substituting $u(x)\equiv c$ into \eqref{ellqap1} we get $p(x)\equiv0$ which
is impossible due to the assumption of $p(x)$. Therefore we complete the proof.
\end{proof}

\medskip

Secondly, we consider a static version of the equation \eqref{Yameq} for $\alpha>1$ and
$\beta>1$ on a smooth metric measure space. Because the proof method of Theorem
\ref{app1} is suitable to the elliptic version of \eqref{Yameq}, we only state the result
without the proof.
\begin{theorem}\label{ellap2}
Let $(M,g,e^{-f}dv)$ be an $n$-dimensional complete smooth metric measure space with
$\mathrm{Ric}_f\geq0$. Assume that there exist three constants $s>0$, $\kappa>0$
and $k>0$ such that $\mu(x)$, $p(x)$ and $q(x)$ in the following elliptic equation
\begin{equation}\label{pro5.3}
\Delta_f u+\mu(x)u+p(x)u^\alpha+q(x)u^\beta=0, \quad \alpha>1,\,\,\beta>1,
\end{equation}
satisfy
\begin{enumerate}
\item $\mu^+\big|_{B(x_0,R)}=o(R^{-s})$\, and\, $\sup_{B(x_0,R)}|\nabla\mu|=o(R^{-s})$,
\, as\, $R\to\infty$;
\item $p^+|_{B(x_0,R)}=o[R^{-\kappa(\alpha-1)}]$ and $\sup_{B(x_0,R)}|\nabla p|=o[R^{-\kappa(\alpha-1)}]$,
\, as\, $R\to\infty$;
\item $q^+|_{B(x_0,R)}=o[R^{-k(\beta-1)}]$ and $\sup_{B(x_0,R)}|\nabla q|=o[R^{-k(\beta-1)}]$,
\, as\, $R\to\infty$.
\end{enumerate}
Let $u(x)$ be positive solution to the elliptic equation \eqref{pro5.3} on $M$ such that
\[
u(x)=o\,(r^{\widetilde{\kappa}}(x)\,)
\]
for some $\widetilde{\kappa}\in(0,l)$ near infinity, where $l:=\min\{\kappa,k\}$. Then $u(x)$
is a positive constant.
\end{theorem}

\medskip

Thirdly, we will apply Theorem \ref{main} to discuss some Yamabe-type problems of
complete Riemannian manifolds and smooth metric measure spaces.

\medskip

We now prove Theorem \ref{app4} by applying Theorem \ref{app1} to the equation \eqref{maineq}.

\begin{proof}[Proof of Theorem \ref{app4}]
In order to prove the theorem, we only need to discuss the nonexistence of positive
smooth solutions $u(x)$ to the equation \eqref{maineq} on $(M,g)$. In Theorem \ref{app1},
if we let
\[
u(x,t)=u(x),\quad f=0,\quad \alpha=\frac{n+2}{n-2},
\]
and
\[
\mu(x)=-\frac{n-2}{4(n-1)}\,\mathrm{S}, \quad
p(x)=\frac{n-2}{4(n-1)}\,\mathrm{\tilde{S}},
\]
then by the assumptions of Theorem \ref{app4}, we know that $\mathrm{S}\ge 0$ and such $u(x)$ does not exist
and hence the theorem follows.
\end{proof}

\medskip
Theorem \ref{generalcase} can be proved by applying Theorem \ref{ellap2} to the equation \eqref{weYaeq}.
\begin{proof}[Proof of Theorem \ref{generalcase}]
Assume that $u$ is a minimizer of the weighted Yamabe constant $\Lambda\le0$. By the proof of Proposition 4.1
in \cite{[Case]}, $u$ is a solution of the equation
\[
\Delta_f u-\frac{m+n-2}{4(m+n-1)}\mathrm{S}^m_fu-c_1e^{\frac fm}u^{\frac{m+n}{m+n-2}}+c_2u^{\frac{m+n+2}{m+n-2}}=0,
\]
where
\begin{align*}
c_1 & = \frac{2m(m+n-1)\Lambda}{n(m+n-2)}\left(\int_M u^{\frac{2(m+n)}{m+n-2}}\right)^{\frac{2m+n-2}{n}}
\left(\int_M u^{\frac{2(m+n-1)}{m+n-2}}e^{\frac fm}\right)^{-\frac{2m+n}{n}} , \\
c_2 & = \frac{(2m+n-2)(m+n)\Lambda}{n(m+n-2)}\left(\int_M u^{\frac{2(m+n)}{m+n-2}}\right)^{\frac{2m-2}{n}}
\left(\int_M u^{\frac{2(m+n-1)}{m+n-2}}e^{\frac fm}\right)^{-\frac{2m}{n}}.
\end{align*}

In order to prove the theorem, we only need to check the nonexistence of nonconstant positive
solutions $u(x)$ to the above equation under the assumptions of Theorem \ref{generalcase}.
Notice that $c_1\le0$ and $c_2\le0$ due to $\Lambda\le0$. In Theorem \ref{ellap2},
if we let
\[
\mu(x)=-\frac{m+n-2}{4(m+n-1)}\mathrm{S}^m_f,\quad p(x)=-c_1e^{\frac fm},\quad q(x)=c_2,
\]
and
\[
\alpha=\frac{m+n}{m+n-2}>1, \quad \beta=\frac{m+n+2}{m+n-2}>1,
\]
then the assumptions of Theorem \ref{generalcase} imply that all the conditions
of Theorem \ref{ellap2} are satisfied and hence such $u(x)$ does not exist,
which contradicts the existence of positive minimizer $u$.
\end{proof}

\medskip

When $\Lambda=0$, we can prove Theorem \ref{app5} by applying Theorem \ref{app2} to
the equation \eqref{weYaeq}.
\begin{proof}[Proof of Theorem \ref{app5}]
Assume that $u$ is a critical point of the weighted Yamabe quotient $\mathcal{Q}(u)$
with $u(x)=e^{o(r^{1/2}(x))}$ near infinity. Then such $u$ satisfies the equation
\eqref{weYaeq}. Since $\Lambda=0$, we have $c_1=c_2=0$ and the critical point in fact
is a minimizer. So \eqref{weYaeq} becomes
\[
\Delta_f u-\frac{m+n-2}{4(m+n-1)}\mathrm{R}^m_fu=0.
\]
In order to prove Theorem \ref{app5}, we only need to check the nonexistence of positive
solutions to the above equation under conditions of Theorem \ref{app5}. Indeed, if we let
\[
\mu(x)=-\frac{m+n-2}{4(m+n-1)}\mathrm{S}^m_f\quad \mathrm{and} \quad
u(x,t)=u(x),
\]
in Theorem \ref{app2}, then the assumptions of Theorem \ref{app5} satisfy all the
conditions of Theorem \ref{app2}. Therefore, by Theorem \ref{app2}, we know that
there does not exist any positive solution $u(x)$ with $u(x)=e^{o(r^{1/2}(x))}$
near infinity. So the theorem follows.
\end{proof}

\

The other goal of this section is that we apply Theorem \ref{mainpara} to study elliptic
Liouville-type theorems for some elliptic versions of the equation \eqref{Yameq} on a
smooth metric measure space. Here, we mainly apply Theorem \ref{mainpara} to
prove Theorem \ref{ellliou} and Corollary \ref{app6}.
\begin{proof}[Proof of Theorem \ref{ellliou}]
Let $u(x)$ be a positive smooth function to the equation
\[
\Delta_f u+p u^\alpha=0, \quad \alpha\leq1,
\]
where $p$ is a nonnegative constant. Since $\mathrm{Ric}^m_f\geq0$ and $a_i=b_i=0$ ($i=1,2,3$),
applying Theorem \ref{mainpara} to this equation, for any $\lambda>1$, we have the gradient estimate
\[
\frac{|\nabla u|^2}{\lambda u^2}+pu^{\alpha-1}\leq\sqrt{\frac{m+n}{2}}\Psi^{\frac12}
\]
by letting $R\to\infty$, where
\[
\Psi:=
\frac{(m+n)\lambda^2\tilde{K}^2}{2(1-\varepsilon)(\lambda-1)^2}
\quad\mathrm{and}\quad
\tilde{K}:=-\frac 12\big[(\alpha-1)(\lambda\alpha-1)p\big]^-\sup_{H_{2R,T}}\{u^{\alpha-1}\}.
\]
In the above estimate, if $\alpha=1$, then $\tilde{K}\equiv0$ and
\[
\frac{|\nabla u|^2}{\lambda u^2}+pu^{\alpha-1}=0,
\]
which implies the theorem. So we only consider the case $\alpha<1$.
In this case, we choose $\lambda=\lambda_0>1$ such that $\lambda_0\alpha<1$
and then $\tilde{K}\equiv0$, hence $\Psi\equiv0$, which also implies the theorem.
\end{proof}

Theorem \ref{ellliou} immediately implies Corollary \ref{app6} as follows.
\begin{proof}[Proof of Corollary \ref{app6}]
Assume that $u$ is a critical point of the weighted Yamabe quotient $\mathcal{Q}(u)$.
Then $u$ satisfies the equation \eqref{weYaeq}. Since $\Lambda=0$, we have $c_1=c_2=0$ and
hence \eqref{weYaeq} becomes
\[
\Delta_f u-\frac{m+n-2}{4(m+n-1)}\mathrm{S}^m_fu=0.
\]
In the following we only need to check the nonexistence of positive solutions to the
above equation under the condition of Corollary \ref{app6}. Indeed, since $\mathrm{S}^m_f$
is nonpositive constant, we know that $-\frac{m+n-2}{4(m+n-1)}\mathrm{S}^m_f$
is nonnegative constant. According to Theorem \ref{ellliou}, we immediately conclude
that there does not exist any nontrivial positive solution to the above equation.
So our assumption does not hold and the theorem follows.
\end{proof}

\bibliographystyle{amsplain}

\begin{thebibliography}{30}
\bibitem{[AB]} D.G. Aronson, P. B\'enilan, R\'egularit\'e des solutions de
l'\'equation des milieux poreux dans $\mathbb{R}^n$, C. R. Acad. Sci. Paris.
S\'er. A-B 288 (1979), A103-A105.

\bibitem{[Aub]}T. Aubin. \'Equations diff\'erentielles non lin\'eaires et probl\`eme de Yamabe
concernant la courbure scalaire, J. Math. Pures Appl. 55 (1976), 269-296.

\bibitem{[Ba]} M. Bailesteanua,  A Harnack inequality for the parabolic Allen-Cahn equation,
 Ann. Global Anal. Geom. 51 (2017), 367-378.

\bibitem{[BCP]} M. Bailesteanua, X.-D. Cao, A. Pulemotov, Gradient estimates
for the heat equation under the Ricci flow, J. Funct. Anal. 258 (2010), 3517-3542.

\bibitem{[BE]}D. Bakry, M. Emery, Diffusion hypercontractivitives,
in: S\'{e}minaire de Probabilit\'{e}s XIX, 1983/1984, in: Lecture
Notes in Math., vol. 1123, Springer-Verlag, Berlin, 1985, pp. 177-206.

\bibitem{[Bis]} S. Bismuth, Prescribed scalar curvature on a complete Riemannian manifold
in the negative case, J. Math. Pures Appl. 79 (2000), 941-951.

\bibitem{[CGS]} L. A. Caffarelli, B. Gidas, J. Spruck, Asymptotic symmetry and local behavior of
semilinear elliptic equations with critical Sobolev growth, Comm. Pure Appl. Math.
42 (1989), 271-297.

\bibitem{[Cala]}E. Calabi, An extension of E. Hopf's maximum principle with
an application to Riemannian geometry, Duke Math. J. 25 (1958), 45-56.

\bibitem{[CaCo]} R.S. Cantrell, C. Cosner, Diffusive logistic equations with indefinite
weights: population models in a disrupted environments, Proc. R. Soc. Edinb. 112A (1989), 293-318.

\bibitem{[Cao]}H.-D. Cao, Recent progress on Ricci solitons, Recent
advances in geometric analysis, Adv. Lect. Math. (ALM) 11, 1-38,
International Press, Somerville, MA 2010.

\bibitem{[CCK]} X.-D. Cao, M. Cerenzia, D. Kazaras, Harnack estimates for the Endangered Species Equation,
Proc. Amer. Math. Soc. 143 (2015), 4537-4545.

\bibitem{[CFL]} X.-D. Cao, B. Fayyazuddin Ljungberg, B. Liu, Differential Harnack estimates for a nonlinear heat
equation. J. Funct. Anal. 265 (2013), 2312-2330.

\bibitem{[CLPW]} X.-D. Cao, B. Liu, I. Pendleton, A. Ward, Differential Harnack estimates for Fisher's equation,
 Pacific J. Math. 290 (2017), 273-300.

\bibitem{[CaZh]}X.-D. Cao, Z. Zhang, Differential Harnack estimates for
parabolic equations, Proceedings of Complex and Differential Geometry, 87-98, 2011.

\bibitem{[Case]}J. Case, A Yamabe-type problem on smooth metric measure spaces, J. Diff.
Geom. 101 (2015), 467-505.

\bibitem{[Cheng-Yau]} S.-Y. Cheng, S.-T. Yau, Differential equations on Riemannian manifolds and their
geometric applications, Comm. Pure Appl. Math. 28 (1975), 333-354.

\bibitem{[ChHam]}B. Chow, R. Hamilton, Constrained and linear Harnack inqualities
for parabolic equations, Invent. Math. 129 (1997), 213-238.

\bibitem{[DKN]}N.-T. Dung, N.-N. Khanh, Q.-A. Ngo, Gradient estimates for some $f$-heat equations
driven by Lichnerowicz's equation on complete smooth metric measure spaces, Manuscripta
Math. 155 (2018), 471-501.

\bibitem{[GS]} B. Gidas, J. Spruck, Global and local behavior of positive solutions of
nonlinear elliptic equations, Comm. Pure Appl. Math. 34 (1981), 525-598.

\bibitem{[GW]} Z.-M. Guo, J.-C. Wei, Hausdoff dimension of ruptures for solutions of a semilinear equation
with singular nonlinearity, Manuscripta Math. 120 (2006), 193-209.

\bibitem{[Ham93]} R. Hamilton, A matrix Harnack estimate for the heat equation, Comm. Anal. Geom.
1 (1993), 113-126.

\bibitem{[Ham]} R. Hamilton, The formation of singularities in the Ricci flow, Surveys in Differential
Geom. 2 (1995), 7-136, International Press.

\bibitem{[Jin]}Z.-R. Jin, A counterexample to the Yamabe problem for complete noncompact manifolds,
Lecture Notes in Math. 1306 (1988), 93-101.

\bibitem{[J93]}Z.-R. Jin, Prescribing scalar curvatures on the conformal classes of complete metrics with
negative curvature, Trans. Amer. Math. Soc. 340 (1993), 785-810.

\bibitem{[LP]} J. Lee, T. Parker, The Yamabe problem, Bull. Amer. Math. Soc. 17 (1987), 37-91.

\bibitem{[LXu]}J.-F. Li, X.-J. Xu, Differential Harnack inequalities on Riemannian
manifolds I: Linear heat equation, Adv. Math., 226 (5) (2011), 4456-4491.

\bibitem{[JLi]}J.-Y. Li, Gradient estimates and Harnack inequalities for nonlinear parabolic and
nonlinear elliptic equations on Riemannian manifolds, J. Funct. Anal. 100 (1991), 233-256.

\bibitem{[LTY]}P. Li, L.-F. Tam, D.-G. Yang, On the elliptic equation $\Delta u+ku-Ku^p=0$ on complete
Riemannian manifolds and their geometric applications, Trans. Amer. Math. Soc. 350 (1998), 1045-1078.

\bibitem{[Li-Yau]}P. Li, S.-T. Yau, On the parabolic kernel of the Schrodinger operator,
Acta Math. 156 (1986), 153-201.

\bibitem{[LD]}X.-D. Li, Liouville theorems for symmetric diffusion operators on complete Riemannian manifolds,
J. Math. Pure. Appl. 84 (2005), 1295-1361.

\bibitem{[Lott1]}J. Lott, Some geometric properties of the Bakry-\'{E}mery-Ricci tensor,
Comment. Math. Helv. 78 (2003), 865-883.

\bibitem{[Lott2]}J. Lott, Remark about scalar curvature and Riemannian submersions, Proc. Amer. Math.
Soc. 135 (2007), 3375-3381.

\bibitem{[MRS]}P. Mastrolia, M. Rigoli, A. G. Setti, Yamabe-type equations on complete,
noncompact manifolds, Progress in Mathematics, vol. 302, Birkh\"auser Verlag, Basel, 2012.

\bibitem{[Per]}G. Perelman, The entropy formula for the Ricci flow and its geometric applications,
arXiv:math.DG/0211159v1, 2002.

\bibitem{[RRV]} A. Ratto, M. Rigoli, L. Veron, Scalar curvature and conformal deformations
of noncompact Riemannian manifolds, Math. Z. 225 (1997), 395-426.

\bibitem{[Sc84]}R. Schoen,  Conformal deformations of a Riemannian metric to constant scalar curvature,
J. Diff. Geom. 20 (1984), 479-495.

\bibitem{[Sch]}R. Schoen, A report on some recent progress on nonlinear problems in geometry, Surveys in
Differential Geometry (Cambridge, MA, 1990), Suppl. No. 1 to J. Differential Geom., Lehigh
Univ., Bethlehem, PA (distributed by Amer. Math. Soc.), 1991, pp. 201-241.

\bibitem{[Sou-Zh]}P. Souplet, Q.S. Zhang,  Sharp gradient estimate and Yau's Liouville theorem for
the heat equation on noncompact manifolds, Bull. London Math. Soc. 38 (2006), 1045-1053.

\bibitem{[Tru]}N. S. Trudinger. Remarks concerning the conformal deformation of Riemannian structures
on compact manifolds,  Ann. Scuola Norm. Sup. Pisa 22 (1968), 265-274.

\bibitem{[WW]}G.-F. Wei, W. Wylie, Comparison geometry for the Bakry-\'{E}mery Ricci tensor,
J. Diff. Geom. 83 (2009), 377-405.

\bibitem{[Wu10]} J.-Y. Wu, Li-Yau type estimates for a nonlinear parabolic equation on complete
manifolds, J. Math. Anal. Appl. 369 (2010) 400-407.

\bibitem{[Wu15]} J.-Y. Wu, Elliptic gradient estimates for a weighted heat equation and applications,
Math. Z. 280 (2015), 451-468.

\bibitem{[Wu17]} J.-Y. Wu, Elliptic gradient estimates for a nonlinear heat equation and applications,
Nonlinear Anal. 151 (2017), 1-17.

\bibitem{[Xu]}X.-J. Xu, Gradient estimates for the degenerate parabolic equation $u_t=F(u)$ on
manifolds and some Liouville-type theorems. J Differential Equations, 252 (2012), 1403-1420.

\bibitem{[Yama]}H. Yamabe, On a deformation of Riemannian structures on compact manifolds, Osaka Math.
J. 12 (1960), 21-37.

\bibitem{[Ya08]}Y.-Y. Yang, Gradient estimates for a nonlinear parabolic equation on Riemannian manifolds,
Proc. Amer. Math. Soc. 136(11) (2008) 4095-4102.

\bibitem{[Yang]}Y.-Y. Yang, Gradient estimates for the equation $\Delta u+cu^{-\alpha}=0$ on Riemannian
manifolds, Acta Math Sin (Engl Ser), 26 (2010), 1177-1182.

\bibitem{[Yau]}S.-T. Yau,  Harmonic functions on complete Riemannian manifolds,
Comm. Pure Appl. Math. 28 (1975), 201-228.

\bibitem{[Zhang]}Q.S. Zhang, Positive solutions to $\Delta u-Vu+Wu^p=0$ and its parabolic counterpart
in noncompact manifolds, Pacific J. Math. 213 (2004), 163-200.

\bibitem{[Zhu1]}X.-B. Zhu, Gradient estimates and Liouville theorems for nonlinear parabolic
equations on noncompact Riemannian manifolds, Nonlinear Anal. 74 (2011), 5141-5146.

\bibitem{[Zhu2]}X.-B. Zhu, Gradient estimates and Liouville theorems for linear and nonlinear
parabolic equations on Riemannian manifolds, Acta Mathematica Scientia, 36B(2) (2016), 514-526.
\end{thebibliography}

\end{document}